\documentclass[aop,MSNbibl,nameyear,seceqn,dvips]{arximspdf}

\doi{10.1214/13-AOP839} 
\volume{42}
\issue{1}
\pubyear{2014}
\firstpage{354}
\lastpage{397}

\makeatletter
\def\bullet{{\mbox{\tiny$\square$}}}

\newcommand{\rrvert}{\vert}
\newcommand{\llvert}{\vert}
\newtheorem{theorem}{Theorem}[section]

\newtheorem{lemma}[theorem]{Lemma}

\newtheorem{corollary}[theorem]{Corollary}

\newcommand{\Z}{\mathbb{Z}}

\newcommand{\R}{\mathbb{R}}

\newcommand{\E}{\mathbf{E}}

\newcommand{\cov}{\operatorname{cov}}

\def\cdt{ \cdot}
\def\ov{\overline}

\def\t{\theta}
\def\D{\Delta}
\def\L{\Lambda}
\def\s{\sigma}
\def\b{\beta}
\def\a{\alpha}
\def\g{\gamma}

\def\P{\mathbf{P} }

\def\O{\mathcal{O}}

\def\na{\mapsto}

\def\={\stackrel{\mathrm{def}}{=}}

\def\ovln#1{ {\overline{ #1}}}
\newcommand{\eqref}[1]{(\ref{#1})}
\makeatother

\begin{document}
\begin{frontmatter}

\title{Explicit rates of approximation in the CLT for quadratic forms}
\runtitle{Explicit rates of approximation in the CLT}

\begin{aug}
\author[A]{\fnms{Friedrich} \snm{G\"otze}\corref{}\thanksref{t1}\ead[label=e1]{goetze@math.uni-bielefeld.de}}
\and
\author[B]{\fnms{Andrei Yu.} \snm{Zaitsev}\thanksref{t1,t2}\ead[label=e2]{zaitsev@pdmi.ras.ru}}
\thankstext{t1}{Supported by the SFB 701 in Bielefeld and by Grant
RFBR-DFG 09-01-91331.}
\thankstext{t2}{Supported by Grants NSh-1216.2012.01, RFBR
09-01-12180, 10-01-00242 and 11-01-12104, by the Alexander von
Humboldt Foundation and by a program of fundamental researches of
Russian Academy of Sciences ``Modern problems of fundamental
mathematics.''}
\runauthor{F.~G\"otze and A.~Yu. Zaitsev}
\affiliation{Universit\"at Bielefeld and St.~Petersburg Department of
Steklov Mathematical Institute}
\address[A]{Fakult\"at f\"ur Mathematik\\
Universit\"at Bielefeld\\
Postfach 100131, D-33501\\
Bielefeld\\
Germany\\
\printead{e1}} 
\address[B]{St.~Petersburg Department\\
\quad of V. A. Steklov Mathematical Institute\\
Fontanka 27\\
St.~Petersburg 191023\\
Russia\\
\printead{e2}}
\end{aug}

\received{\smonth{12} \syear{2011}}
\revised{\smonth{2} \syear{2013}}

%
\begin{abstract}
Let $X, X_1, X_2, \ldots$ be {i.i.d.}
${\mathbb R}^d$-valued real random vectors. Assume that ${\mathbf{E}
X=0}$, $\cov X =\mathbb C$,
$\mathbf{E} \Vert X\Vert^2=\s^2 $ and
that $X $ is
not concentrated in a proper subspace
of $\mathbb R^d$. Let $G $ be a mean zero Gaussian random vector
with
the same covariance operator as that of~$X$.
We study the distributions of nondegenerate quadratic forms
$ \mathbb Q [S_N ] $ of the normalized sums~${S_N=N^{-1/2} (X_1+\cdots+X_N)} $
and show that,
without any additional conditions,
\[
\Delta_N\= \sup_x \bigl|\mathbf{P} \bigl\{ \mathbb
Q [S_N]\leq x \bigr\}- \mathbf{P} \bigl\{ \mathbb Q [G ]\leq x
\bigr\} \bigr| = {\mathcal O} \bigl(N^{-1} \bigr),
\]
provided that $d\geq5 $ and
the fourth moment of $X $ exists.
Furthermore, we provide explicit bounds of order~${\mathcal O}
(N^{-1} )$
for $\Delta_N$ for the rate of approximation by short asymptotic
expansions and for the concentration functions of the random
variables $\mathbb Q [S_N+a ]$, $a\in{\mathbb R}^d$. The order of the
bound is optimal.
It extends previous
results of Bentkus and G\"otze [\emph{Probab. Theory Related Fields}
\textbf{109} (1997a)
367--416]
(for ${d\ge9}$) to the case $d\ge5$,
which is the smallest possible dimension for such a bound.
Moreover, we show that, in {the} finite
dimensional case and for isometric~$ \mathbb Q $, the {implied}
constant in
${\mathcal O} (N^{-1} )$ has the form $c_d
\s^d (\det\mathbb C)^{-1/2} \E\|\mathbb C^{-1/2} X\|^4$ with
some $c_d$ depending on $d$ only.
This answers a long standing question about optimal rates in
the central limit theorem for quadratic forms starting with a
seminal paper by Ess\'een [\emph{Acta Math.} \textbf{77} (1945) 1--125].
\end{abstract}

%

\end{frontmatter}

\section{Introduction}
\label{s1}

Let $\mathbb R^d $ be the $d$-dimensional space of real vectors
$x=(x_1,\ldots,x_d) $ with scalar product $\langle x,y\rangle
=x_1y_1+\cdots+x_dy_d$ and norm $\|x\|= \langle x,x\rangle^{1/2}$.
We also denote by $\mathbb R^\infty$ a separable Hilbert space
consisting of all real sequences ${x=(x_1,x_2,\ldots)}$ such that
$\|x\|^2= x_1^2+x_2^2+\cdots<\infty$.

Let $X,X_1,X_2,\ldots$ be a sequence of i.i.d. $\mathbb R^d$-valued
random vectors. Assume that ${\E X= 0}$ and $\s^2\=\E
\|X\|^2<\infty$. Let $G$ be a mean zero Gaussian random vector
such that its covariance operator $\mathbb C= \cov G\dvtx\mathbb R^d
\to
\mathbb R^d $
is equal to $\cov X$. It is well known that the distributions
$\mathcal L(S_N)$ of sums
%
%
\begin{equation}
\label{SN}S_N\=N^{-1/2} ( X_1+\cdots+
X_N)
\end{equation}
converge weakly to $\mathcal L(G)$.

Let $\mathbb Q\dvtx\mathbb R^d\to\mathbb R^d$ be a linear symmetric
bounded operator, and let $\mathbb Q [x]= \langle\mathbb Q
x,x\rangle$ be the corresponding quadratic form. We say that
$\mathbb Q $ is nondegenerate if $\ker\mathbb Q= \{
0 \}$.

Denote, for $q> 0$,
\[
\b_q\=\E\|X\|^q,\qquad \b\=\b_4.
\]
Introduce the distribution functions
%
%
\begin{equation}
F (x) \= \P\bigl\{ \mathbb Q [S_N]\leq x \bigr\}, \qquad H(x) \= \P
\bigl\{ \mathbb Q [G]\leq x \bigr\}.\label{eq11j}
\end{equation}
Write
%
%
\begin{equation}
\label{eqdel}\D_N\= \sup_{x\in\mathbb R} \bigl| F(x) - H(x) \bigr|.
\end{equation}

%
\begin{theorem}{\label{T11}} Assume that
$\mathbb Q $ and $\mathbb C $ are nondegenerate and that
$d\geq5$ or $d = \infty$.
Then
\[
\D_N \leq c ( \mathbb Q, \mathbb C ) \b/N.
\]
The constant $c(\mathbb Q, \mathbb C )$ in this bound depends
on $ \mathbb Q $ and $ \mathbb C $ only.
\end{theorem}

%
\begin{theorem}{\label{T11a}} Let the conditions of
Theorem \ref{T11} be satisfied, and let $5\le d<\infty$.
Assume that the operator $\mathbb Q$ is isometric. Then
\[
\D_N \leq c_d \s^d (\det\mathbb
C)^{-1/2} \E\bigl\|\mathbb C^{-1/2} X\bigr\|^4 /N.
\]
The constant $c_d$ in this bound depends
on $d$ only.
\end{theorem}

Theorems \ref{T11} and \ref{T11a} are simple consequences of
the main result of this paper, Theorem \ref{T15}; see also
Theorem \ref{T13}. Theorem \ref{T11} was proved in G\"otze and
Zaitsev (\citeyear{GotZai08}). It confirms a conjecture of Bentkus and G{\"o}tze
(\citeyear{BenGot97N1}) [below BG (\citeyear{BenGot97N1})]. It
generalizes to the case $d\ge5$ the
corresponding result of BG (\citeyear{BenGot97N1}). In their Theorem
1.1, it was
assumed that $d\ge9$, while our Theorem \ref{T11} is proved for
$d\ge5$. Theorem \ref{T11a} yields an explicit bound in terms
of the distribution $\mathcal L(X)$.

The distribution function of $ \|S_N\|^2 $ (for bounded $X$
with values in $\mathbb R^d$) may have jumps of order $\O
(N^{-1} )$, for all $1\leq d\leq\infty$; see, for example,
BG [(\citeyear{BenGot97N1}), page~468]. Therefore, the bounds of
Theorems \ref{T11}
and \ref{T11a} are optimal with respect to the order in $N$.

Theorems \ref{T11}, \ref{T11a} and the method of their proof are closely
related to the lattice point problem in number theory. Suppose
that $d<\infty$ and that $\langle\mathbb Q x,x\rangle>0$,
for $x\ne0$.
Let $\mathrm{vol} E_r $ be the volume of the ellipsoid
\[
E_r = \bigl\{ x\in\mathbb R^d\dvtx\mathbb Q [x] \leq
r^2 \bigr\} \qquad\mbox{for } r\geq0.
\]
Write $\mathrm{vol}_{\mathbb Z} E_r $ for the number of points in
$E_r\cap\mathbb Z^d$, where $ \mathbb Z^d \subset\mathbb\R^d$ is the
standard lattice of
points with integer coordinates.

The following result due to G\" otze (\citeyear{Got04}) is related to Theorems
\ref{T11} and \ref{T11a}; see also BG (\citeyear{BenGet95,BenGot97N2}).

%
%
\begin{theorem}\label{T12}
For all dimensions $ d\geq5$,
\[
\sup_{a\in\mathbb R^d} \biggl\llvert\frac{ \mathrm{vol}
_{\mathbb Z} (E_r+a) -\mathrm{vol} E_r }{ \mathrm{vol} E_r
} \biggr\rrvert= \O
\bigl(r^{-2}\bigr)\qquad \mbox{for } r\geq1,
\]
where the constant in $\O(r^{-2})$ depends on the dimension
$d$ and on the lengths of axes of the ellipsoid $E_1$ only.
\end{theorem}

Theorem \ref{T12} solves the lattice point problem for $d\geq5$.
It improves the classical estimate $ \O( r^{-2d/(d+1) } ) $ due to
\citet{Lan15}, just as Theorem \ref{T11} improves the bound
${\O(N^{-d/(d+1)})}$ by Ess\'een (\citeyear{Ess45}) in the CLT for ellipsoids
with axes parallel to coordinate axes. A related result for
indefinite forms may be found in G\" otze and Margulis (\citeyear{GotMar10}).

Work on the
estimation of the rate of approximation under the
conditions of Theorem \ref{T11} for Hilbert spaces
started in
the second half of the last century. See Zalesski\u\i, Sazonov and
Ulyanov (\citeyear{ZalSazUly88}) and \citet{Nag89} for
{optimal bounds} of order $\O(N^{-1/2})$
(with respect to eigenvalues of $\mathbb C$)
assuming finiteness of the third moment. For a more
detailed discussion see Yurinskii (\citeyear{Yur82}), Zalesski\u\i,
Sazonov and
Ulyanov (\citeyear{ZalSazUly89}),
Bentkus, G\"otze, Paulauskas and Ra\v ckauskas (\citeyear{Benetal91}),
BG {(\citeyear{BenGet95N2,BenGot96,BenGot97N1})} and
Senatov (\citeyear{Sen97,Sen98}).

Under some more restrictive moment
and dimension conditions the estimate of order $\O(N^{-1+\varepsilon})$,
with $\varepsilon\downarrow0$ as $d\uparrow\infty$, was obtained by
G\"
otze (\citeyear{Got79}). The proof in G\" otze (\citeyear{Got79}) was
based on a new
symmetrization inequality for characteristic functions of
quadratic forms. This inequality is related to Weyl's (\citeyear{Wey15})
inequality for trigonometric sums. This inequality and its
extensions (see Lemma~\ref{L51}) play
a crucial role in the
proofs of bounds in the CLT
for ellipsoids and hyperboloids in
finite and infinite dimensional cases. Under some additional
smoothness assumptions, error bounds $\O(N^{-1})$ (and, moreover,
Edgeworth type expansions) were obtained in G\" otze (\citeyear{Got79}),
\citet{Ben84}, Bentkus, G\" otze and Zitikis (\citeyear
{BenGotZit93}). BG {(\citeyear{BenGet95N2,BenGot96,BenGot97N1})}
established the bound of order $\O(N^{-1})$
without smoothness-type conditions. Similar bounds for the rate of
infinitely divisible approximations were obtained by Bentkus, G\"
otze and Zaitsev (\citeyear{BenGotZai97}). Among recent publications,
we should
mention the papers of \citeauthor{NagChe99} (\citeyear{NagChe99,NagChe05})
(${d\ge13}$, providing a more precise dependence of constants on
the eigenvalues of $\mathbb C$) and Bogatyrev, G\"otze and
Ulyanov (\citeyear{BogGotUly06}) (nonuniform bounds for $d\ge12$);
see also G\"otze
and Ulyanov (\citeyear{GotUly00}). The proofs of bounds of order $\O
(N^{-1})$ are
based on discretization (i.e., a reduction to lattice valued
random vectors) and the symmetrization techniques mentioned above.

Assuming
the matrices $\mathbb Q$ and $\mathbb C$ to be diagonal, and the
independence of the first five coordinates of $X$, BG (\citeyear{BenGot96}) have already reduced the dimension
requirement for
the bound $\O(N^{-1})$ to $d\ge5$. The independence assumption in
BG (\citeyear{BenGot96}) allowed to apply an adaption of the Hardy--Littlewood
circle method. For the general case considered in
Theorem \ref{T11},
one needs to develop new techniques. Some yet
unpublished results of G\"otze (\citeyear{Go94}) provide the rate
$\mathcal
O(N^{-1})$ for sums of two independent \textit{arbitrary} quadratic
forms (each of rank $d \ge3$). G\"otze and Ulyanov (\citeyear{GotUly03})
obtained bounds of order $\mathcal O(N^{-1})$ for some ellipsoids
in $\mathbb R^d$ with $d\ge5$ in the case of lattice distributions of $X$.

The optimal possible dimension condition for this rate is just
$d\geq5$, due to the lower bounds of order $\mathcal O(N^{-1}\log{
N})$ for dimension $d=4$ in the corresponding lattice point
problem. The question about precise convergence rates in
dimensions
$2\leq d \leq4$ still remains completely open (even in
the simplest case where $\mathbb Q$ is the identity operator $\mathbb
I_d$, and for random
vectors with
independent Rademacher coordinates). It should be mentioned that,
in the case $d=2$, a precise convergence rate would imply a
solution of the famous circle problem. Known lower bounds in the
circle problem correspond to the bound of
order $\mathcal O (N^{-3/4} \log^\delta
N )$, $\delta>0$, for $\Delta_N$. Hardy (\citeyear{Har}) conjectured
that up to
logarithmic factors this is the optimal order.

Now we describe the most important elements of the proof. We have
to mention that a big part of the proof repeats the arguments of
BG (\citeyear{BenGot97N1}); see BG (\citeyear{BenGot97N1}) for the
description and application of
symmetrization inequality and
discretization procedure.
In our proof we do not use the multiplicative inequalities of BG
(\citeyear{BenGot97N1}).
Here we replace those techniques
by arguments from the geometry of numbers, developed in
G\"otze (\citeyear{Got04}), combined with
effective equidistribution results by G\"otze and Margulis
(\citeyear{GotMar10}) for suitable actions of
unipotent subgroups of $\operatorname{SL}(2,\R)$; see Lemma \ref{GM}.
These new techniques (compared to previous results) are mainly
concentrated in Sections~\mbox{\ref{s4}--\ref{s7}}.

Using the Fourier inversion formula [see \eqref{eq31} and
\eqref{eq31o}], we have to estimate some integrals of the
absolute values of differences of characteristic functions of
quadratic forms. In Section \ref{s5}, we reduce the estimation of
characteristic functions to the estimation of a theta-series; see
Lemma \ref{L75} and inequality \eqref{koren}. To this end, we
write the expectation with respect to Rademacher random variables
as a sum with binomial weights $p(m)$ and $p(\ov m)$. Then we
estimate $p(m)$ and $p(\ov m)$ from above by discrete Gaussian
exponential weights $c_s q(m)$ and $c_s q(\ov m)$; see~\eqref{qm}, \eqref{pm}, \eqref{eq79} and \eqref{eq710}. Together
with the nonnegativity of some characteristic functions [see
\eqref{eq78} and \eqref{eq713}], this allows us to apply then
the Poisson summation formula from Lemma \ref{Le32}. This formula
reduces the problem to an estimation of integrals of theta-series.
Section \ref{s6} is devoted to some facts from number theory. We
consider the lattices, their $\alpha$-characteristics [which are
defined in \eqref{alp} and \eqref{alp3}] and Minkowski's
successive minima. In Section \ref{s7}, we reduce the estimation
of integrals of theta-series to some integrals of
$\alpha$-characteristics. An application of the crucial Lemma \ref{GM},
mentioned above,
ends the proof.

\section{Results}
\label{s11}

To formulate the results we need more notation repeating most
part of the notation used in BG (\citeyear{BenGot97N1}). Let $\s
_1^2\geq
\s_2^2\geq\cdots$ be the eigenvalues of~$\mathbb C$, counting
their multiplicities. We have $\s^2=\s_1^2+\s_2^2+\cdots$.

We identify the linear operators and corresponding matrices. By
$\mathbb I_d\dvtx\mathbb R^{d}\to\mathbb R^{d}$ we denote the identity
operator and, simultaneously, the diagonal matrix with entries 1
on the diagonal. By $\mathbb O_{d}$ we denote the $(d\times d)$
matrix with zero entries.

Throughout $ \mathcal S=\{ {e}_{1},\ldots, {e}_{s} \}\subset\mathbb
R^d$ denotes a finite
set of cardinality $s$. We write $\mathcal S_o$ instead of
$\mathcal S$ if the system $\{ {e}_{1},\ldots, {e}_{s} \}$ is
orthonormal. Let $p>
0$ and $\delta\geq0$. Denote
%
%
\begin{equation}
P(\delta,\mathcal S, Y)=\min_{e\in
\mathcal S} \P\bigl\{ \|Y -e\|\leq\delta
\bigr\}. \label{eq13uuu}
\end{equation}
Similarly to BG (\citeyear{BenGot97N1}),
we use the following nondegeneracy condition for the distribution
of a $d$-dimensional vector $Y$:
%
%
\begin{equation}
P_{\mathbb Q}(\delta,\mathcal S, Y)\=\min\bigl\{P(\delta,\mathcal S,
Y),P(\delta,\mathbb Q\mathcal S, Y) \bigr\} \geq p, \label{eq1.3}
\end{equation}
where $p>0$ is a parameter involved in the condition. Note that
%
%
\begin{equation}
\label{eq13}P(\delta,\mathcal S, Y) =P_{\mathbb
I_d}(\delta,\mathcal S, Y).
\end{equation}

Introduce truncated random vectors
%
%
\begin{eqnarray}
\label{eq14a}\qquad  X^\diamond&= &X {\mathbf I}\bigl \{ \|X\|\leq\s \sqrt{N} \bigr\}
,\qquad
X_\diamond=X {\mathbf I} \bigl\{ \|X\|> \s \sqrt{N} \bigr\},
\\
\label{eq14t} X^\bullet&=& X {\mathbf I} \bigl\{\bigl\| \mathbb
C^{-1/2} X\bigr\|\leq\sqrt{dN} \bigr\},\qquad X_\bullet=X {\mathbf I}
\bigl\{ \bigl\|\mathbb C^{-1/2} X\bigr\|> \sqrt{dN} \bigr\},
\end{eqnarray}
and their
moments (for $q>0$)
%
%
\begin{eqnarray}
\L_4^\diamond&=& \frac{ 1 }{ \s^{4} N } \E\bigl\|X^\diamond
\bigr\|^4,\qquad \Pi_q^\diamond= \frac{ N }{ (\s
\sqrt{N})^{q} } \E
\|X_\diamond\|^q,\label{eq15}\\
\L_4^\bullet&=& \frac{ 1 }{ d^{2} N } \E\bigl\|\mathbb C^{-1/2}
X^\bullet\bigr\|^4,\qquad \Pi_q^\bullet=
\frac{ N }{ ( \sqrt{dN})^{q} } \E\bigl\|\mathbb C^{-1/2} X_\bullet\bigr\|^q.
\label{eq15t}
\end{eqnarray}
Here and below ${\mathbf I} \{ A \} $ denotes the indicator
of an
event $A$. Of course, definitions~\eqref{eq14t} and \eqref{eq15t}
have sense if $d<\infty$ and the covariance operator
$\mathbb C $ is nondegenerate.

Clearly, we have
%
%
\begin{equation}
\label{eq14u} X^\diamond+X_\diamond=X^\bullet+X_\bullet=X,\qquad
\bigl\Vert X^\diamond\bigr\Vert\Vert X_\diamond\Vert=\bigl\Vert X^\bullet
\bigr\Vert\Vert X_\bullet\Vert=0.
\end{equation}
Generally speaking, $X^\bullet$ and $X^\diamond$ are
different truncated vectors. In BG (\citeyear{BenGot97N1}) the i.i.d.
copies of
the vectors $X^\diamond$ and $X_\diamond$ only were involved.
Truncation \eqref{eq14t} was there applied to the vector
$X^\diamond$. The use of $X^\bullet$ is more natural for
the estimation of constants in the case $d<\infty$.
It is easy to see that
%
%
\begin{equation}
\label{eq14rr} \bigl(\mathbb C^{-1/2} X \bigr)^{\diamond}= \bigl(
\mathbb C^{-
1/2} X \bigr)^{\bullet}= \mathbb C^{-1/2}
X^{\bullet}
\end{equation}
and
%
%
\begin{equation}
\label{eq14ru} \bigl(\mathbb C^{-1/2} X \bigr)_{\diamond}= \bigl(
\mathbb C^{-
1/2} X \bigr)_{\bullet}= \mathbb C^{-1/2}
X_{\bullet}.
\end{equation}
Equalities \eqref{eq14rr} and \eqref{eq14ru} provide a
possibility to apply auxiliary results obtained in BG (\citeyear
{BenGot97N1}) for
truncated vectors $X^\diamond$ and $X_\diamond$ to truncated
vectors $\mathbb C^{-1/2} X^{\bullet}$ and $\mathbb C^{-
1/2} X_{\bullet}$. However, one should take into account that
$\s^2$, $\L_4^\diamond$, $\Pi_q^\diamond$, $G$, $\ldots$ have to
be replaced by corresponding objects related to the vector $\mathbb
C^{-1/2} X$ (i.e., by
$d$, $\L_4^\bullet$, $\Pi_q^\bullet$, $\mathbb
C^{- 1/2} G, \ldots$).

By $c, c_1,c_2,\ldots$ we denote absolute positive constants. If
a constant depends on, say, $s$, then we point out the
dependence writing $c_s$ or $c(s)$. We denote by $c$ universal
constants which might be different in different places of the
text. Furthermore, in the conditions of theorems and lemmas (see,
e.g., Theorem \ref{T13} and the proofs of Theorems \ref{T15},
\ref{T16} and \ref{T21}) we write $c_0$ for an \textit{arbitrary}
positive absolute constant; for example, one may choose $c_0=1$.
We write $A\ll B$ if there exists an
absolute constant $c$ such that $A\leq c B$. Similarly, $A\ll_s
B$ if $A\leq c(s) B$. We also write $A\asymp_s B$ if $A\ll_s
B\ll_s A$. By $ \lfloor\a\rfloor$ we denote the largest integer
not greater than $\a$.

Throughout we assume that all random vectors and variables are
independent in aggregate if the contrary is not clear from the
context.
By $X_1,X_2,\ldots$ we
shall denote independent copies of a random vector $X$.
Similarly, $G_1,G_2,\ldots$ are
independent copies of $G$ and so on. By $\mathcal
L(X)$ we denote the distribution of $X$. Define the symmetrization
$\widetilde X $ of a random vector $X$ as a random vector with
distribution $ {\mathcal L (\widetilde X)= \mathcal L (X_1-X_2)}$.

Instead of normalized sums $S_N$, it is sometimes more convenient
to consider the sums $Z_N= {X}_{1}+\cdots+ {X}_{N} $. Then
$S_N=N^{-1/2} Z_N$.
Similarly, by $Z_N^{(\diamond)}$ (resp., $Z_N^{(\bullet)}$) we
shall denote sums of $N$ independent copies of $X^\diamond$ (resp.,~$X^\bullet$).
For example,
$Z_N^{(\bullet)}={X_1^\bullet}+\cdots+X_N^\bullet$.

The expectation ${\mathbf E}_Y$
with respect to a random vector $Y$
we define as the conditional expectation
\[
{\mathbf E}_Y f(X,Y,Z,\ldots) = \E\bigl( f(X,Y,Z\ldots) | X,Z,
\ldots\bigr)
\]
given all
random vectors but $Y$.

Throughout we write
$\operatorname{e}\{ x \}\=
\exp\{i x \}$.
By
%
%
\begin{equation}
\label{Fur}\widehat F(t)=\int_{-\infty}^\infty
\operatorname{e}\{tx\} \,dF(x),
\end{equation}
we denote the Fourier--Stieltjes transform
of a function $F$ of bounded variation or, in other words, the
Fourier transform of the measure which has the distribution
function $F$.

Introduce the distribution functions
%
%
\begin{eqnarray}\label{eq11}
F_a (x) \= \P\bigl\{ \mathbb Q [S_N-a]\leq x \bigr
\},\qquad  H_a(x) \= \P\bigl\{ \mathbb Q [G-a]\leq x \bigr\},
\nonumber
\\[-8pt]
\\[-8pt]
 \eqntext{a
\in\mathbb R^d, x\in\mathbb R.}
\end{eqnarray}
Furthermore, define, for $d=\infty$ and $a\in\mathbb R^d$, the Edgeworth
correction
\[
E_a(x)=E_a(x; \mathbb Q, X )
\]
as
a function of bounded variation such that
$E_a(-\infty) =0 $ and its
Fourier--Stieltjes transform
is given by
%
%
\begin{eqnarray}\label{eq12}
\widehat E_a(t)= \frac{ 2 (it)^2 }{ 3 \sqrt{N} } \E\operatorname{e}
\bigl\{t
\mathbb Q [Y] \bigr\} \bigl( 3 \langle\mathbb Q X, Y \rangle\langle
\mathbb Q
X,X\rangle+2 i t \langle\mathbb Q X, Y \rangle^3 \bigr),
\nonumber
\\[-8pt]
\\[-8pt]
\eqntext{Y= G-a.}
\end{eqnarray}

In finite dimensional spaces (for $1\le d<\infty$) we define the
Edgeworth correction as follows; see Bhattacharya and Rao (\citeyear
{BhaRan86}).
Let $\phi$ denote the standard normal density in $\mathbb R^d$. Then
$p(y) =\phi(\mathbb C^{-1/2} y)/\sqrt{\operatorname{det}
\mathbb C } $, $y\in\mathbb R^d$, is the density of~$G$, and, for
$a\in\mathbb R^d$, $b=\sqrt N a$, we have
%
%
\begin{eqnarray}\label{eq121}
E_a(x) &\=& \Theta_b (Nx) \= \frac{ 1 }{ 6 \sqrt{N} }
\chi(A_x),
\nonumber
\\[-8pt]
\\[-8pt]
\nonumber
 A_x& = &\bigl\{ u\in\mathbb
R^d\dvtx\mathbb Q [u-a]\leq x \bigr\},
\end{eqnarray}
with the signed measure
%
%
\begin{equation}
\chi(A) \=\int_A \E p'''(y)
X^3 \,dy \qquad\mbox{for the Borel sets } A\subset\mathbb
R^d,\label{eq122}
\end{equation}
and where
%
%
\begin{equation}
p'''(y) u^3 = p(y)
\bigl( 3 \bigl\langle\mathbb C^{-1} u,u\bigr\rangle\bigl\langle
\mathbb
C^{-1} y,u\bigr\rangle- \bigl\langle\mathbb C^{-1} y,u\bigr
\rangle^3 \bigr) \label{eq123}
\end{equation}
denotes the third Frechet derivative of $p$
in direction $u$.

Notice
that $E_a =0$
if $a=0$ or if ${\E\langle X,y\rangle^3=0}$, for all $y\in
\mathbb R^d$. In particular, $E_a =0$ if $X $ is symmetric
[i.e., $\mathcal L(X)=\mathcal L(-X)$].

We can write similar representations for $E_a^{\bullet}(x) =
\Theta_b^{\bullet} (Nx)$ and $E_a^{\diamond}(x) =
\Theta_b^{\diamond} (Nx)$ just replacing $X$ by $X^\bullet$ and
$X^\diamond$ in \eqref{eq12} or \eqref{eq122}
with $Y=G-a$.

For $b\in\R^d$, introduce the distribution functions
%
%
\begin{equation}
\Psi_b(x) \= \P\bigl\{ \mathbb Q [Z_N-b] \leq x
\bigr\}=F_a(x/N)\label{eq118}
\end{equation}
and
%
%
\begin{equation}
\Phi_b(x) \= \P\bigl\{ \mathbb Q [\sqrt{N} G -b]\leq x \bigr
\}=H_a(x/N).\label{eq119}
\end{equation}
Define, for $a\in\mathbb R^d$, $b=\sqrt N a$,
%
%
\begin{equation}
\label{edg} \D_N^{(a)}\=\sup_{x\in\mathbb R} \bigl|
F_a (x) -H_a(x)-E_a(x) \bigr| =\sup
_{x\in\mathbb R} \bigl|\Psi_b (x) - \Phi_b(x)-
\Theta_b(x) \bigr|;\hspace*{-35pt}
\end{equation}
see \eqref{eq11},
\eqref{eq121}, \eqref{eq118} and \eqref{eq119} to justify the
last equality in \eqref{edg}. We write $\D_{N,\bullet}^{(a)}$ and
$\D_{N,\diamond}^{(a)}$ replacing $E_a$ by $E_a^{\bullet}$
and $E_a^{\diamond}$ in \eqref{edg}.

The aim of this paper is to derive for $\D_N^{(a)}$ explicit
bounds of order $\mathcal O(N^{-1})$ without any additional
smoothness type assumptions. Theorem \ref{T13} [which was proved
in BG (\citeyear{BenGot97N1})] solved this problem in the case $13\leq
d\leq
\infty$.

In Theorems \ref{T13}--\ref{T21} we assume that the symmetric
operator $\mathbb Q$ is isometric, that is, that $\mathbb Q^2$ is
the identity operator $\mathbb I_d$. This does not restrict
generality; see Remark~1.7 in BG (\citeyear{BenGot97N1}). Indeed, any symmetric
operator $\mathbb Q $ may be decomposed as $\mathbb Q=\mathbb
Q_1\mathbb Q_0\mathbb Q_1$, where $\mathbb Q_0 $ is symmetric
and isometric and $\mathbb Q_1 $ is symmetric bounded and
nonnegative, that is, $\langle\mathbb Q_1 x,x\rangle\geq0$,
for all $x\in\mathbb R^d$. Thus, for any symmetric $\mathbb Q$, we can
apply
all our bounds replacing the random vector
$X $ by $\mathbb Q_1X, $ the Gaussian random vector $G $
by $\mathbb Q_1G$, the shift $a $ by $\mathbb Q_1a$, etc. In
the case of concentration functions (see Theorems \ref{T16} and
\ref{T21}), we have
${Q(X; \lambda; \mathbb Q )=
Q(\mathbb Q_1X; \lambda; \mathbb Q_0 )}$, and we may apply the
results provided $\mathbb Q_1X $ (instead of $X$) satisfies the
conditions.

%
\begin{theorem}[{[BG (\citeyear{BenGot97N1}), Theorem
1.3]}]\label{T13} Assume
that
$\delta= 1/300$, $\mathbb Q^2=\mathbb I_d$, ${s=13}$ and $13\leq
d\leq\infty$. Let $ P_{\mathbb Q}(\delta, \mathcal
S_o,c_0 G/\s) \geq p>0$, where $c_0$ is an arbitrary positive absolute
constant. Then
%
%
\begin{equation}
\D_N^{(a)}\leq C \bigl( \Pi_3^{\diamond}
+ \L_4^{\diamond} \bigr) \bigl( 1 + \Vert a/\s
\Vert^6 \bigr) \label{eq16}
\end{equation}
and
%
%
\begin{equation}
\D_{N,\diamond}^{(a)}\leq C \bigl( \Pi_2^{\diamond}
+ \L_4^{\diamond} \bigr) \bigl( 1 + \Vert a/\s
\Vert^6 \bigr) \label{eq16w}
\end{equation}
with $C=c p^{-6} + c (\s/\t_{8})^{8}$, where $\t_1^4 \geq
\t_2^4 \geq\cdots$ are the eigenvalues of $(\mathbb C \mathbb
Q)^2$.
\end{theorem}

Unfortunately, we cannot apply Theorem \ref{T13} for
$d=5,6,\ldots,12$. Moreover, the quantity $C$ depends on $p$
which is exponentially small with respect to eigenvalues of
$\mathbb C$.

The main result of the paper is Theorem \ref{T15}. It is valid
for $5\le d<\infty$ in finite-dimensional spaces $\mathbb R^d $ only.
However, the bounds of Theorem \ref{T15}
depend on the smallest $\s_j$'s. This makes them
unstable if one or more of coordinates of $X$ degenerates. In our
finite dimensional results, Theorems
\ref{T15}, \ref{T16}, \ref{T21} and Corollary \ref{T15c},
we always assume that the covariance operator
$\mathbb C $ is nondegenerate.

%
\begin{theorem}{\label{T15}} Let
$\mathbb Q^2=\mathbb I_d$, $5\le d<\infty$.
Then
%
%
\begin{equation}
\D_{N}^{(a)} \leq C \bigl( \Pi_3^\bullet+
\L_4^\bullet\bigr) \bigl( 1 + \Vert a/\s\Vert^3
\bigr) \label{eq18}
\end{equation}
and
%
%
\begin{equation}
\D_{N,\bullet}^{(a)} \leq C \bigl( \Pi_2^\bullet+
\L_4^\bullet\bigr) \bigl( 1 + \Vert a/\s\Vert^3
\bigr), \label{eq18w}
\end{equation}
with $C= c_d \s^d (\det\mathbb C)^{-1/2}$.
\end{theorem}

In G\"otze and Zaitsev (\citeyear{GotZai10}) [see also a preprint of
G\"otze and
Zaitsev (\citeyear{GotZai09}) which is available in Internet], an
analogue of
Theorem \ref{T15} was proved in the case $s=5$ and ${5\le
d<\infty}$ with bounds for constants which are not optimal. It
extends to the case $d\ge5$ Theorem 1.5 of BG (\citeyear{BenGot97N1}) which
contains the corresponding bounds for $d\ge9$. Unfortunately, in
both papers, the quantity $C$ depends on $p$ which is
exponentially small with respect to $\s_9/\s^2$ [in BG (\citeyear
{BenGot97N1})]
and to $\s_5/\s^2$ [in G\"otze and Zaitsev (\citeyear{GotZai10})].
Under some
additional conditions, $C$ may be estimated from above by
$c_d \exp(c\s^2\s_9^{-2})$ and by
$c_d \exp(c\s^2\s_5^{-2})$, respectively. The case $a=0$ was
considered earlier in G\"otze and Zaitsev (\citeyear{GotZai08}). As a
consequence, we have proved Theorem~\ref{T11}.

It is easy to see that, according to \eqref{eq14t} and \eqref{eq15t},
%
%
\begin{equation}
\label{eq18fc}\quad \Pi_3^\bullet+ \L_4^\bullet
\le\E\bigl\| \mathbb C^{-1/2} X\bigr\|^{3+\delta}/\bigl(d^{(3+\delta)/2}
N^{(1+\delta)/2}\bigr)\qquad \mbox{for } 0\le\delta\le1
\end{equation}
and
%
%
\begin{equation}
\label{eq18y}\Pi_2^\bullet+ \L_4^\bullet
\le\E\bigl\|\mathbb C^{-
1/2} X\bigr\|^{2+\delta}/\bigl(d^{(2+\delta)/2}
N^{\delta/2}\bigr)\qquad \mbox{for } 0\le\delta\le2.
\end{equation}
Therefore, Theorem \ref{T15} implies the following Corollary \ref{T15c}.

%
\begin{corollary}{\label{T15c}} Let
$\mathbb Q^2=\mathbb I_d$, $5\le d<\infty$.
Then
%
%
\begin{equation}\qquad
\D_{N}^{(a)} \ll_d C \bigl( 1 + \Vert a/\s
\Vert^3 \bigr) \E\bigl\|\mathbb C^{-1/2} X\bigr\|^{3+\delta}/N^{(1+\delta)/2}
\qquad \mbox{for } 0\le\delta\le1 \label{eq18c}
\end{equation}
and
%
%
\begin{equation}\qquad
\D_{N,\bullet}^{(a)} \ll_d C \bigl( 1 + \Vert a/\s
\Vert^3 \bigr) \E\bigl\|\mathbb C^{-
1/2} X\bigr\|^{2+\delta}/N^{\delta/2}\qquad
\mbox{for } 0\le\delta\le2, \label{eq18wc}
\end{equation}
with $C= \s^d (\det\mathbb C)^{-1/2}$. In particular,
%
%
\begin{equation}
\label{eq18f}\max\bigl\{\D_{N}^{(a)},\D
_{N,\bullet
}^{(a)} \bigr\} \ll_d C \bigl( 1 + \Vert
a/\s\Vert^3 \bigr) \E\bigl\|\mathbb C^{-1/2} X\bigr\|^4/N.
\end{equation}
\end{corollary}

Theorem \ref{T13} and Corollary \ref{T15c} yield Theorems \ref
{T11} and
\ref{T11a},
using that $E_0(x)\equiv0$, $\E\|\mathbb C^{-1/2} X\|^4\le\b/\s
_d^4$, and
$\Pi_2^{\diamond}+\L_4^{\diamond}\leq\Pi_3^{\diamond}+\L
_4^{\diamond
}\leq
\b/(\s^4 N)$.

Comparing Theorem \ref{T15} and Corollary \ref{T15c} with the
main results of BG (\citeyear{BenGot97N1}) and G\"otze and Zaitsev
(\citeyear{GotZai10}), we see
that the constants in Theorem \ref{T15} and Corollary \ref{T15c}
are written explicitly in terms of moment characteristics of
$\mathcal L(X)$. In the case of nonpositive definite quadratic
forms $\mathbb Q$ such kind of estimates were unknown.

If, in the conditions of Theorem \ref{T15}, the distribution of
$X$ is symmetric or $a=0$, then the Edgeworth corrections
$E_a(x)$ and $E_a^\bullet(x)$ vanish and
%
%
\begin{equation}\label{eq18ff}
\D_{N}^{(a)} =\D_{N,\bullet}^{(a)}\leq C
\bigl( \Pi_2^\bullet+ \L_4^\bullet\bigr)
\bigl( 1 + \Vert a/\s\Vert^3 \bigr),\qquad C=
c_d \s^d (\det\mathbb C)^{-1/2}.\hspace*{-35pt}
\end{equation}
The corresponding
inequality from Theorem 1.4 of BG (\citeyear{BenGot97N1}) yields in
the case $s=9$
and $9\leq d\leq\infty$ under the condition $ P_{\mathbb Q}(
\delta, \mathcal S_o,c_0 G/\s) \geq p>0 $ with $\delta= 1/300$ the
bound
%
%
\begin{equation}
\D_N^{(a)}\leq C \bigl( \Pi_2^{\diamond}
+ \L_4^{\diamond}\bigr) \bigl(1+\Vert a/\s\Vert^{4 }
\bigr),\qquad C=c p^{-4}. \label{eq18fg}
\end{equation}
It is clear that sometimes the bound
\eqref{eq18fg} may be sharper than \eqref{eq18ff} but,
unfortunately, it depends on $p$ which is usually exponentially
small with respect to $\s_9/\s^2$.

Several authors have obtained more precise estimates of constants
in the case of $d$-dimensional balls with $d\ge12$, including the case
$d=\infty$. For balls,
$\mathbb Q=\mathbb I_d$. In the papers mentioned above, the
authors have used the aproach of BG (\citeyear{BenGot97N1}) and
obtained bounds
with constants depending on $s\le d$ largest eigenvalues
$\s_1^2\ge\s_2^2\ge\cdots\ge\s_s^2$ of the covariance
operator $\mathbb C$; see \citeauthor{NagChe99} (\citeyear{NagChe99,NagChe05}),
with $d\ge s=13$, and G\"otze and Ulyanov (\citeyear{GotUly00}), and Bogatyrev,
G\"otze and Ulyanov (\citeyear{BogGotUly06}), with $d\ge s=12$. It should be
mentioned, that, in a particular case, where $\mathbb Q=\mathbb
I_d$ and $d\ge12$, these results may be sharper than \eqref{eq18},
for some covariance operators $\mathbb C$.
The lower bounds for $\Delta^{(a)}_N$ under different conditions on $a$
and $\mathcal L(X)$
are given in G\"otze and Ulyanov (\citeyear{GotUly00}). See the upper
bounds for $\Delta^{(a)}_N$ with $s=12$ and $d=\infty$ in Ulyanov and
G\"otze (\citeyear{Uly11}), where the dependence on the eigenvalues of
$\mathbb C$
is given
in the upper bound in an explicit form which
coincides with that in the lower bound. See also the review
of recent results for ``almost'' quadratic forms in \citet{ProUly13}.

Thus we see that the statement of Theorem \ref{T15} is
especially interesting for $d=5,\ldots,11$. It is new even in
the case of $d$-dimensional balls. It is plausible that the bounds
for constants in Theorem \ref{T15} could be also improved for
balls with $d\ge5$, especially in the case where $d$ is large.
It seems, however, that this is
impossible in the case of general $\mathbb Q$ even if $\mathbb
Q^2=\mathbb I_d$. For example, consider the operator $\mathbb Q$
such that $\mathbb Q e_j=e_{d-j+1}$, where $\mathbb C
e_j=\s_j^2e_j$, $j=1,2,\ldots,d$, are eigenvectors of $\mathbb C$.
Following the proof of Theorem \ref{T15}, we see that the bounds
for the modulus of the characteristic function $ | \widehat\Psi_b
(t) |= |\E\operatorname{e} \{ t \mathbb Q [Z_N
-b] \} |$ behave as the bounds for the modulus of the
characteristic function $ |\E\operatorname{e} \{ t
\mathbb I_d [Z_N -b] \} |$, but with eigenvalues of the
covariance operator $\s_1\s_d$, $\s_2\s_{d-1}$,
$\s_3\s_{d-2}$, \ldots which may be essentially smaller than
$\s_1^2\ge\s_2^2\ge\s_3^2\ge\cdots$. Therefore, it is natural that
the bounds for constants in Theorem \ref{T15} depends on the
smallest eigenvalues of the covariance operator $\mathbb C$.

Note that, in the proof of Theorem \ref{T13} in BG (\citeyear{BenGot97N1}),
inequalities \eqref{eq16} and~\eqref{eq16w} were derived for the
Edgeworth correction $E_a(x)$ defined by \eqref{eq12}. However,
from Theorems \ref{T13} and \ref{T15} it follows that, at least
for $13\le d<\infty$, definitions \eqref{eq12} and \eqref{eq121}
determine the same function $E_a(x)$. Indeed, both functions may
be represented as $N^{-1/2} K_j(x)$, where $K_j(x)$ are some
functions of bounded variation which are independent of $N$.
Furthermore, inequalities \eqref{eq16} and \eqref{eq18} provide
both bounds of order $\mathcal O(N^{-1})$. This is possible only if
the Edgeworth corrections $E_a(x)$ are the same in these
inequalities.

On the other hand, it is proved (for $d\geq9$) that definition
\eqref{eq12} determines
a function of bounded variation [see BG (\citeyear{BenGot97N1}, Lemma
5.7)], while definition~\eqref{eq121} has no sense for ${d=\infty}$.

Introduce the concentration function
%
%
\begin{eqnarray}
\label{con} Q(X; \lambda)&=&Q(X; \lambda; \mathbb Q )
\nonumber
\\[-8pt]
\\[-8pt]
\nonumber
& = &\sup
_{a\in\mathbb R^d, x\in\R} \P\bigl\{ x \leq\mathbb Q [X-a]\leq
x+\lambda\bigr\}\qquad \mbox{for } \lambda\geq0.
\end{eqnarray}
Note that, evidently, $Q(X+Y; \lambda)\le Q(X;
\lambda)$, for any $Y$ which is independent of~$X$.

We say that a random vector $Y$ is concentrated in $\mathbb
L\subset\mathbb R^d $ if $\P\{ Y\in\mathbb L\} =1$. In BG [(\citeyear
{BenGot97N1}), item
(iii) of Theorem 1.6] it was shown that if
$\widetilde X$ is not concentrated in a proper closed linear
subspace of $\mathbb R^d$, $1\leq d\leq\infty$, then for any $\delta
>0$ and
$\mathcal S$,
there exists a natural number $m$ such that the condition
$ P_{\mathbb Q}(\delta,\mathcal S, m^{-1/2} \widetilde Z_m ) \geq p$
holds with
some $ p>0$.

In this paper, we shall prove the following Theorems \ref{T16}
and \ref{T21}.

%
\begin{theorem}{\label{T16}} Let $\mathbb Q^2=\mathbb I_d$,
$5\leq s= d< \infty$ and $0\leq\delta\leq1/(5 s)$. Then:
\begin{longlist}[(ii)]
\item[(i)]
%
%
\begin{equation}
 Q(Z_N; \lambda)\ll_d (p N)^{-1} \max\bigl\{
1; \lambda\s^{-2} \bigr\} \s^d (\det\mathbb
C)^{-1/2}\qquad\mbox{for all } \lambda\geq0, \label{eq19}\hspace*{-35pt}
\end{equation}
if $ P( \delta
,\mathcal S_o, \mathbb C^{- 1/2} \widetilde X )\ge p$ for some $\mathcal S_o$ and $p>0$.

\item[(ii)]
%
%
\begin{equation}
 Q(Z_N; \lambda)\ll_d (p N)^{-1} \max\bigl\{
m; \lambda\s^{-2} \bigr\} \s^d (\det\mathbb
C)^{-1/2} \qquad \mbox{for all } \lambda\geq0, \label{eq110}\hspace*{-35pt}
\end{equation}
if, for
some $\mathcal S_o$ and positive integer $m$,
$ P( \delta,\mathcal S_o, m^{-1/2} \mathbb C^{-
1/2} \widetilde Z_m )\ge p>0$.
\end{longlist}
\end{theorem}

%
\begin{theorem}{\label{T21}} Assume that $5\leq d< \infty$
and that $\mathbb Q^2=\mathbb I_d$.
Then
%
%
\begin{eqnarray}\label{eq21}
Q(Z_N; \lambda)\ll_d \max\bigl\{ \Pi_2^\bullet+
\L_4^\bullet; \lambda \s^{-2} N^{-1}
\bigr\} \s^d (\det\mathbb C)^{-1/2}
\nonumber
\\[-8pt]
\\[-8pt]
\eqntext{\mbox{for all } \lambda\geq
0.}
\end{eqnarray}
In particular,
$Q(Z_N; \lambda)\ll_d N^{-1} \max\{ \E\|\mathbb C^{-
1/2} X\|^4; \lambda\s^{-2} \} \s^d (\det\mathbb C)^{-1/2}$.
\end{theorem}

Theorems \ref{T16} and \ref{T21} yield more explicit
versions of Theorems 1.5 and 2.1 from G\"otze and Zaitsev (\citeyear{GotZai10})
[which extend to the case $5\le d\le\infty$ Theorems~1.6 and 2.1
of BG (\citeyear{BenGot97N1}) which were proved for $9\le d\le\infty
$]. We should
mention that the results of G\"otze and Zaitsev (\citeyear{GotZai10})
do not
follow from Theorems
\ref{T15}, \ref{T16} and \ref{T21}.
For example, they may be sharper than Theorems
\ref{T15}, \ref{T16} and \ref{T21}, in a particular case, where
$\mathbb Q=\mathbb I_d$ and
$\s_5\asymp_d\s$. Under some additional conditions, $\s^d (\det
\mathbb
C)^{-1/2}$ is replaced by $\exp(c\s^2\s_5^{-2})\asymp_d1$. On the
other hand, $\s^d (\det\mathbb
C)^{-1/2}$ provides a power-type dependence on eigenvalues of $\mathbb C$
and the results are valid for $\mathbb Q$ which might be not
positive definite.

In Theorems \ref{T15} and \ref{T21}, we do not assume
conditions $ P(\cdt)\ge p>0$ or $ P_{\mathbb Q}(\cdt)\ge p>0$. In
the proofs, we use, however, that, for any fixed absolute positive
constant $c_0$ and any positive quantity $c_d$ depending on $d$
only, condition $ P(\delta,\mathcal S_o, c_0 \mathbb C^{- 1/2}
G)\ge p$ is fulfilled with $s=d$, $\delta=c_d$ and $p\asymp_{d}1$, for
any orthonormal system $\mathcal S_o$.

Similarly to BG
(\citeyear{BenGot97N1}), in Section \ref{s2}, we prove bounds for
concentration
functions. The proof is technically simpler as that of Theorem
\ref{T15}, but it shows how to apply the principal ideas. This
proof repeats almost literally the corresponding proof of BG
(\citeyear{BenGot97N1}). The only difference consists in the use of
new Lemma
\ref{GZ2} which allows us to estimate characteristic functions of
quadratic forms for relatively large values of argument $t$. In
Sections \ref{s3} and \ref{s4}, Theorem \ref{T15} is proved. We
replace Lemma 9.4 of BG (\citeyear{BenGot97N1}) by its improvement,
Lemma \ref{L94}. Another difference is in another choice of $k$
in \eqref{eq156a} and \eqref{eq156c} in comparison with that in
BG (\citeyear{BenGot97N1}). In Sections \ref{s5}--\ref{s7} we prove
estimates for
characteristic functions which were discussed in
Section~\ref{s1}.\looseness=-1

\section{Proofs of bounds for concentration functions}
\label{s2}

\subsection*{Proof of Theorems \protect\ref{T16} and \protect\ref{T21}} Below we
prove assertions \eqref{eq19};
\eqref{eq19}${}\Longrightarrow$ \eqref{eq110} and \eqref{eq110}${}\Longrightarrow{}$\eqref{eq21}.
The proof repeats almost literally the
corresponding proof of BG (\citeyear{BenGot97N1}). It is given here for
the sake of completeness. The only essential difference is in the use
of Lemma \ref{GZ2} in the proof of Lemma \ref{L23}. We have also to
replace everywhere 9 by 5 and $\diamond$ by $\bullet$.\vadjust{\goodbreak} 

For $ 0\leq t_0\leq T$ and $b\in\R^d$,
define the integrals
\[
I_0=\int_{-T}^{T} \bigl| \widehat
\Psi_b (t)\bigr | \,dt, \qquad I_1= \int_{ t_0 \leq|t|\leq T}
\bigl| \widehat\Psi_b (t) \bigr| \frac{ dt }{ | t | },
\]
where
%
%
\begin{equation}
\label{eq823} \widehat\Psi_b(t)=\E\operatorname{e} \bigl\{ t
\mathbb Q [Z_N -b] \bigr\}
\end{equation}
denotes the Fourier--Stieltjes
transform of
the distribution function $\Psi_b $
of $ \mathbb Q [Z_N-b]$. Note that $ | \widehat\Psi_b
(-t) |= | \widehat\Psi_b (t) |$.

%
\begin{lemma}\label{L23} Assume
that $ P(\delta, \mathcal S_o, \mathbb C^{-1/2} \widetilde X)\ge p>0$
with some $
0\leq\delta\leq1/(5 s)$ and $5\le s=d<\infty$. Let $\s^2=1$ and
%
%
\begin{equation}
\label{eq23} t_0=c_1(s) \s_1^{-2}(p
N)^{-1+ 2/s},\qquad  c_2(s) \s_1^{-2}\leq
T\leq c_3(s) \s_1^{-2}
\end{equation}
with some positive constants $c_j(s)$,
$1\leq j\leq3$. Then
%
%
\begin{equation}
I_0\ll_s (\det\mathbb C)^{-1/2} (p
N)^{-1},\qquad I_1 \ll_s (\det\mathbb
C)^{-1/2} ( p N)^{-1}. \label{eq2.3}
\end{equation}
\end{lemma}

\begin{pf}
Note that the condition $\s^2=1$ implies that
%
%
\begin{equation}
\label{MTq}T\asymp_s\s_1^2
\asymp_s\s^2=1 \quad\mbox{and}\quad \det\mathbb C
\leq 1.
\end{equation}
Denote $k =p N$. Without loss of
generality we assume that
$k\geq c_s$, for a sufficiently large quantity $c_s$ depending on $s$ only.
Indeed, if $k\leq c_s$, then one can prove~\eqref{eq2.3} using
\eqref{MTq} and $|\widehat\Psi_b |\leq1$. Choosing $c_s$ to be large
enough, we ensure that $k\geq c_s$ implies $1/k\leq t_0 \leq T$.

Lemma \ref{GZ2} and \eqref{MTq} imply now that
%
%
\begin{equation}
\label{MTN38} \int_{c_4(s)k^{-1+2/s}}^{T } \bigl| \widehat
\Psi_b (t) \bigr| \frac{ dt }{ t } \ll_s \frac{ (\det\mathbb C)^{-1/2}
}{ k }
\end{equation}
for any $c_4(s)$ depending on $s$ only. Inequalities
\eqref{MTq} and \eqref{MTN38} imply \eqref{eq2.3} for $I_1$.

Let us prove inequality \eqref{eq23} for $I_0$.
By \eqref{MTq} and by Lemma \ref
{GZ},
for any $\g>0$ and any fixed $t\in\mathbb R$
satisfying $k^{1/2}\llvert t\rrvert\le c_5(s)$, where $c_5(s)$ is
an arbitrary quantity depending on $s$ only, we have (taking into
account that $|\widehat\Psi_b |\leq1$)
%
%
\begin{equation}
\label{MTN}\quad  \bigl| \widehat\Psi_b (t)\bigr | \ll_{\g, s} \min\bigl
\{ 1; k^{-\g}+ k^{-s/2} \llvert t\rrvert^{-s/2} (\det
\mathbb C)^{-1/2} \bigr\}, \qquad k= p N.
\end{equation}
Furthermore, choosing an appropriate $\g$ and using \eqref
{MTq}--\eqref
{MTN}, we obtain
%
%
\begin{equation}
\label{MTN23}(\det\mathbb C)^{1/2} I_0 \ll_s
\int_{0}^{ 1/k} {dt} + \frac{ 1 }{ k } +\int
_{1/k}^\infty\frac{ dt }{ ( t k )^{s/2} } \ll_s
\frac{ 1 }{ k },
\end{equation}
proving \eqref{eq23} for $I_0$.
\end{pf}

\begin{pf*}{Proof of \eqref{eq19}}
Let $\s^2=1$. Using a well-known
inequality for concentration functions [see, e.g., Petrov
(\citeyear{Pet75}), Lemma 3 of Chapter 3], we have
%
%
\begin{equation}
Q(Z_N; \lambda)\leq4 \sup_{b\in\mathbb R^d} \max\{ {
\lambda}; 1 \} \int_{0}^1 \bigl| \widehat
\Psi_b (t) \bigr| \,dt.\label{eq25}
\end{equation}
To
estimate the integral in \eqref{eq25} we apply Lemma
\ref{L23} which implies that
%
%
\begin{equation}
\label{MTN89}Q(Z_N; \lambda)\ll_d \max\{ {\lambda}; 1
\} (p N)^{-1} (\det\mathbb C)^{-1/2},
\end{equation}
proving \eqref{eq19}
in the case $\s^2=1$. If $\s^2\ne1$, we obtain \eqref{eq19}
applying \eqref{MTN89} to $Z_N/\s$.
\end{pf*}

\begin{pf*}{Proof of \eqref{eq19}${}\Longrightarrow{}$\eqref{eq110}}
Without loss of generality we can assume that $N/m\geq2$. Let
$Y_1,Y_2,\ldots$ be independent copies of $m^{-1/2} Z_m$.
Denote $W_k= Y_1+\cdots+Y_k$. Then $\mathcal L (Z_N)= \mathcal L
(\sqrt{m} W_k +y)$,
where $k=\lfloor N/m\rfloor$ is the
largest integer not greater than $N/m$ and
$y$ is independent of $W_k$. Therefore,
$Q(Z_N; \lambda) \leq Q(W_k; \lambda/m )$. In order to estimate
$Q(W_k;
\lambda/m )$ we apply \eqref{eq19} replacing $Z_N$ by~$W_k$. We
have
%
%
\begin{eqnarray}
Q(W_k; \lambda/m )&\ll_s &(p k )^{-1}
\max\bigl\{ 1; \lambda\s^{-2} /m \bigr\} \s^d (\det\mathbb
C)^{-1/2}
\nonumber
\\[-8pt]
\\[-8pt]
\nonumber
&\ll&(p N )^{-1} \max\bigl\{ m; \lambda\s^{-2} \bigr\}
\s^d (\det\mathbb C)^{-1/2}.
\end{eqnarray}
\upqed\end{pf*}

Recall that truncated random vectors and their moments are
defined by \eqref{eq14a}--\eqref{eq15t} and that $\mathbb C=\cov
X=\cov G$.

%
\begin{lemma}\label{L24}
The random vectors $X^\bullet$, $X_\bullet$ satisfy
\[
\langle\mathbb C x,x\rangle= \bigl\langle\cov X^\bullet x,x\bigr
\rangle+\E\langle X_\bullet,x\rangle^2+ \bigl\langle\E
X^\bullet,x\bigr\rangle^2.
\]
There exist independent centered Gaussian vectors $G_\ast$ and
$W$ such that
\[
\mathcal L (G) =\mathcal L (G_\ast+W )
\]
and
\[
2 \cov G_\ast=2 \cov X^\bullet=\cov\widetilde{X}^\bullet,\qquad
\langle\cov W x,x\rangle=\E\langle X_\bullet,x\rangle^2+
\bigl\langle\E X^\bullet,x\bigr\rangle^2.
\]
Furthermore,
\[
\E\bigl\| \mathbb C^{-1/2} G \bigr\|^2= d=\E\bigl\| \mathbb
C^{-1/2} G_\ast\bigr\|^2 + \E\bigl\|\mathbb C^{-1/2}
W \bigr\|^2
\]
and
$\E\|\mathbb C^{-1/2} W \|^2\leq2 d \Pi_2^\bullet$.
\end{lemma}
We omit the simple proof of this lemma; see BG [(\citeyear{BenGot97N1}), Lemma
2.4] for the same statement with $\diamond$ instead of $\bullet$.

Recall that $Z_N^{(\bullet)}$ and $Z_N^{(\diamond)}$ denote sums
of $N$ independent copies of $X^\bullet$ and $X^\diamond$,
respectively.

%
\begin{lemma}\label{L25} Let $\varepsilon>0$.
There exist absolute positive constants $c$ and $c_1$ such that
the condition $\Pi_2^\bullet\leq c_1 p \delta^2 / (d\varepsilon^2)$
implies that
\[
P\bigl( \delta, \mathcal S, \varepsilon\mathbb C^{-1/2} G \bigr
)\ge p\quad
\Longrightarrow \quad P \bigl( 4 \delta, \mathcal S, \varepsilon(2
m)^{-1/2} \mathbb C^{-1/2} \widetilde{Z}_m^{(\bullet)}
\bigr)\ge p/4
\]
for
$m\geq c\varepsilon^4 d^2 N \L_4^\bullet/ (p \delta^4 )$.
\end{lemma}

Lemmas \ref{L24} and \ref{L25} are in fact the statements of Lemmas
2.4 and 2.5 from BG~(\citeyear{BenGot97N1}) applied to the vectors
$\mathbb
C^{-1/2} X$ instead of the vectors $X$. We use in this connection
equalities \eqref{eq13}, \eqref{eq14rr} and \eqref{eq14ru}
replacing in the formulation $\s^2$, $\L_4^\diamond$,
$\Pi_q^\diamond$, $G$, $Z_m^{(\diamond)}$, $\ldots$
by $d$, $\L_4^\bullet$, $\Pi_q^\bullet$, $\mathbb C^{-1/2} G$,
$Z_m^{(\bullet)}, \ldots,$ respectively.

\begin{pf*}{Proof of \eqref{eq110}${}\Longrightarrow{}$\eqref{eq21}} By
a standard truncation argument, we have
%
%
\begin{equation}\quad
\bigl| \P\{ Z_N \in A \} - \P\bigl\{ Z_N^{(\bullet)} \in
A \bigr\} \bigr| \leq N \P\bigl\{ \bigl\|\mathbb C^{-1/2} X\bigr\|> \sqrt{dN}
\bigr\}
\leq\Pi_2^\bullet\label{eq28}
\end{equation}
for any Borel set $A$, and
%
%
\begin{equation}
Q(Z_N,\lambda) \leq\Pi_2^\bullet+ Q
\bigl(Z_N^{(\bullet)},\lambda\bigr). \label{eq29}
\end{equation}
Recall
that we are proving \eqref{eq21} assuming that $5\le d<\infty$.
It is easy to see that, for any absolute positive constant $c_0$
and for any orthonormal system $ \mathcal S_o=\{ {e}_{1},\ldots,
{e}_{s} \}\subset\mathbb R^d$, condition
%
%
\begin{equation}\qquad
P\bigl(\delta,\mathcal S_o, c_0 \mathbb
C^{- 1/2} G\bigr)\ge p \quad\mbox{with } p\asymp_d1, 5\le s=d<
\infty, \delta=1/(20 s)\label{eq33ss}
\end{equation}
is in fact
fulfilled automatically since the vector $\mathbb C^{-1/2} G$
has standard Gaussian distribution in $\mathbb R^d$ and, therefore,
\[
\P\bigl\{\bigl \Vert c_0 \mathbb C^{-1/2} G-e\bigr\Vert\le\delta
\bigr\}=\P\bigl\{ \bigl\Vert\mathbb C^{-1/2} G-c_0^{-1}
e\bigr\Vert\le c_0^{-1} \delta\bigr\}=c(d,c_0)
\]
for any vector $e\in\mathbb R^d$ with $\Vert e\Vert=1$. For any fixed $c_0$,
the $c(d,c_0)$ may be considered as a quantity depending on $d$
only. Clearly, $4\delta=1/(5s)$. Write $K= \varepsilon/ \sqrt{2} $
with $\varepsilon= c_0 $.
Then, by \eqref{eq33ss} and Lemma \ref{L25}, we have
%
%
\begin{equation}
P \bigl( \delta,\mathcal S_o,\varepsilon \mathbb
C^{-1/2} G \bigr)\ge p \quad\Longrightarrow \quad P \bigl( 4 \delta,\mathcal
S_o,m^{-1/2} K \mathbb C^{-1/2}
\widetilde{Z}_m^{(\bullet)} \bigr)\ge p/4, \label{eq210}\hspace*{-35pt}
\end{equation}
provided that
%
%
\begin{equation}
\Pi_2^\bullet\leq c_1(d),\qquad m\geq
c_2(d) N \L_4^\bullet.\label{eq211}
\end{equation}
Without loss of generality
we may assume that $\Pi_2^\bullet\leq c_1(d)$, since otherwise
the result follows easily from the trivial inequality
$Q(Z_N;\lambda)\leq1$.

The nondegeneracy
condition \eqref{eq210} for $K \widetilde{Z}_m^{(\bullet)} $
allows us to apply inequality~\eqref{eq110}
of Theorem \ref{T16}, and, using \eqref{eq33ss}, we obtain
%
%
\begin{eqnarray}\label{eq212}
Q\bigl(Z_N^{(\bullet)},\lambda\bigr) &=& Q\bigl( K
Z_N^{(\bullet)},K^2 \lambda\bigr)
\nonumber
\\[-8pt]
\\[-8pt]
\nonumber
&\ll_d& N^{-1} \max\bigl\{ m; K^2
\lambda/K^2\s^2 \bigr\} \s^d (\det\mathbb
C)^{-1/2}
\end{eqnarray}
for any $m$
such that \eqref{eq211} is fulfilled. Choosing the minimal $m$ in
\eqref{eq211}, we obtain
%
%
\begin{equation}
Q\bigl(Z_N^{(\bullet)},\lambda\bigr) \ll_d \max\bigl
\{ \L_4^\bullet; \lambda/\bigl(\s^2 N\bigr)\bigr
\} \s^d (\det\mathbb C)^{-1/2}. \label{eq213}
\end{equation}
Combining the estimates
\eqref{eq29} and \eqref{eq213}, we complete the proof.
\end{pf*}


\section{Auxiliary lemmas}
\label{s3}

In Sections \ref{s3} and \ref{s4} we prove Theorem
\ref{T15}. Therefore, we assume that its conditions are
satisfied. We consider the case $d<\infty$
assuming that the following conditions are
satisfied:
%
%
\begin{equation}
\mathbb Q^2=\mathbb I_d,\qquad \s^2 =1,\qquad d\ge5,\qquad
b=\sqrt N a.\label{eq33}
\end{equation}

Notice that the assumption $\s^2=1$ does not restrict generality
since from Theorem \ref{T15} with $\s^2=1$, we can derive the
general result
replacing $X$, $G$ by $X/\s$,
$G/\s$, etc. Other assumptions in \eqref{eq33} are included as
conditions in Theorem \ref{T15}.
Section \ref{s3} is devoted to some auxiliary lemmas which are similar to
corresponding lemmas of BG (\citeyear{BenGot97N1}).

In several places, the proof of Theorem \ref{T15} repeats almost
literally the proof of Theorem 1.5 in BG (\citeyear{BenGot97N1}).
Note, however,
that we use truncated vectors $X^\bullet_j$,\vspace*{-1pt} while in BG (\citeyear
{BenGot97N1})
the vectors $X^\diamond_j$ were involved. We start with
an application of the Fourier transform
to the functions $\Psi_b$ and $\Phi_b$, where $b=\sqrt N a$. We
estimate integrals over the Fourier transforms
using results of Sections \ref{s2}, \ref{s5}--\ref{s7} and some technical
lemmas of BG (\citeyear{BenGot97N1}).
We also apply some methods of estimation of the rate of
approximation in the CLT in multidimensional spaces; cf., for example,
Bhattacharya and Rao (\citeyear{BhaRan86}).

Below we use the following formula for the Fourier inversion; see,
for example, BG (\citeyear{BenGot97N1}). A smoothing inequality of
\citet{Pra72}
implies [see BG (\citeyear{BenGot96}), Section 4] that
%
%
\begin{equation}
F(x) = \frac{ 1 }{ 2 } + \frac{ i }{ 2\pi} \operatorname{V.P.} \int
_{|t|\le K} \operatorname{e} \{ -xt \} \widehat F (t)
\frac{ dt }{ t } +R \label{eq31}
\end{equation}
for any $K>0$ and any distribution
function $F$ with characteristic function~$\widehat F $ [see~\eqref{Fur}], where
%
%
\begin{equation}
\label{eq31o}|R|\leq\frac{ 1 }{ K } \int_{|t|\le K}\bigl |
\widehat F (t)\bigr |\, {dt}.
\end{equation}
Here $ \operatorname{V.P.} \int f(t)
\,dt =\lim_{\varepsilon\to0} \int_{|t| > \varepsilon}f(t) \,dt$
denotes the principal value of the integral.

In Sections \ref{s3} and \ref{s4}, we denote
%
%
\begin{equation}
X'= X^\bullet-\E X^\bullet+W,\label{eq120}
\end{equation}
where $W$ is a centered
Gaussian random vector which is independent of all other random
vectors and variables and is chosen so that $\cov X'= \cov G$.
Such a vector~$W$ exists by Lemma \ref{L24}. We define
$E_a^{\prime}(x) = \Theta_b^{\prime} (Nx)$ replacing $X$ by
$X^\prime$ in~\eqref{eq12} or \eqref{eq122} with $Y=G-a$.

Recall that the random vector $X^\bullet$ is defined in
\eqref{eq14t} and $Z_N^{(\bullet)}$ is a sum
of its $N$ independent copies. Similarly, $Z_N'={X_1'}+\cdots+X_N'$.
Write $\Psi^\bullet_b $
and $\Psi'_b $ for the distribution function of $\mathbb Q
[Z_N^{(\bullet)}-b ]$ and $\mathbb Q [Z_N'-b ]$, respectively.
For $0\leq k\leq N$ introduce the distribution function
%
%
\begin{equation}
\Psi^{(k)}_b (x) =\P\bigl\{ \mathbb Q \bigl[
{G}_{1}+\cdots+ {G}_{k} + X_{k+1}'
+\cdots+ X_{N}' -b \bigr]\leq x \bigr\}.
\label{eq316}
\end{equation}
Notice that
$\Psi^{(0)}_b=\Psi'_b$, $\Psi^{(N)}_b=\Phi_b $.

The proof of the following lemma repeats the proof of Lemma 3.1 of BG
(\citeyear{BenGot97N1}).
The difference is that here we use the truncated vectors $X_j^\bullet$ instead
of~$X_j^\diamond$.

%
\begin{lemma}\label{L31}Let $c_d$ be
a quantity depending on $d$ only. There exist positive quantities
$c_1(d)$ and $c_2(d)$ depending on $d$ only such that the
following statement is valid.
Let $\Pi_2^\bullet\leq c_1(d) p $
and let
an integer $1\leq m\leq N$
satisfy ${m\geq c_2(d) N \L_4^\bullet/p}$. Write
\[
K= c_0^2 /(2 m),\qquad t_1 =
c_d (p N/m )^{-1+2/d}.
\]
Let $F$ denote any of the
functions $\Psi^\bullet_b$, $\Psi'_b$, $\Psi^{(k)}_b$ or
$\Phi_b$. Then we have
%
%
\begin{equation}
F (x) = \frac{ 1 }{ 2 } + \frac{ i }{ 2\pi} \operatorname{V.P.}
\int
_{|t|\leq t_1 } \operatorname{e} \{ -x t K \} \widehat F (t K )
\frac{ dt }{ t } +R_1, \label{eq317}
\end{equation}
with $
|R_1|\ll_d (p N )^{-1} m (\det\mathbb C)^{-1/2}$.
\end{lemma}

\begin{pf}
 We assume that $ (p N )^{-1} m\le c_3(d)$ with
sufficiently small $c_3(d)$ since otherwise the statement of
Lemma \ref{L31} is trivial; see \eqref{MTq}, \eqref{eq31} and
\eqref{eq31o}. Let us prove \eqref{eq317}. We combine
\eqref{eq31} and Lemma \ref{L23}. Changing the variable $t= \tau
K $
in formula \eqref{eq31}, we obtain
%
%
\begin{equation}
F (x) = \frac{ 1 }{ 2 } + \frac{ i }{ 2\pi} \operatorname{V.P.}
\int
_{|t|\leq1} \operatorname{e} \{ -x t K \} \widehat F (t K )
\frac{ dt }{ t } +R, \label{eq319}
\end{equation}
where
%
%
\begin{equation}
\llvert R\rrvert\leq\int_{|t|\leq1 } \bigl| \widehat F (t K ) \bigr|\,dt. \label{eq319a}
\end{equation}
Notice that $\Psi_b^\bullet$,
$\Psi'_b$, $\Psi_b^{(k)}$ and $\Phi_b$ are distribution
functions of random variables which may be written in the following form:
%
%
\[
\mathbb Q [V+T ], \qquad V\= G_1 +\cdots+ G_k
+X_{k+1}^\bullet+\cdots+X_N^\bullet,
\]
with some $k$, $0\leq k\leq N$, and some random vector $T$ which
is independent of $X_j^\bullet$ and $G_j$, for all $j$. Let us
consider separately two possible cases, $k\geq N/2$ and
$k<N/2$.\vadjust{\goodbreak}

\textit{The case $k< N/2$}. Let
$Y$ denote a sum of $m$ independent copies of
$K^{1/2} X^\bullet$. Let ${Y}_1, {Y}_2,\ldots$
be independent copies of $Y$.
Then
we have
%
%
\begin{equation}
\mathcal L\bigl(K^{1/2} V\bigr) =\mathcal L( {Y}_{1}+
\cdots+ {Y}_{l} +T_1) \label{eq320}
\end{equation}
with $l= \lfloor N/(2 m)\rfloor$ and some
random $T_1$ independent of $ {Y}_{1},\ldots, {Y}_{l} $. By \eqref
{eq33ss} and by
Lemma \ref{L25}, we have
%
%
\begin{equation}
P\bigl( \delta, \mathcal S, c_0 \mathbb C^{-1/2} G
\bigr)\ge p\quad  \Longrightarrow\quad  P\bigl( 4 \delta, \mathcal S, \mathbb
C^{-1/2} \widetilde Y \bigr)\ge p/4\label{eq321}
\end{equation}
provided that
%
%
\begin{equation}
\Pi_2^\bullet\ll p/d^{3} \quad\mbox{and}\quad m\gg
d^6N \L_4^\bullet/ p.\label{eq322}
\end{equation}
The inequalities in \eqref{eq322} follow
from conditions of Lemma \ref{L31} if we choose some sufficiently
small (resp., large) $c_1(d)$ [resp. $c_2(d)$]. Due to
\eqref{eq33ss}, \eqref{eq33}, \eqref{eq320} and~\eqref{eq321},
we can apply Lemma \ref{L23} in order to estimate the integrals
in \eqref{eq319} and~\eqref{eq319a}.
Replacing in Lemma \ref{L23} $X$ by $Y$ and $N$ by $l$, we obtain
\eqref{eq317} in the case $k<N/2$.

\textit{The case $k\geq N/2$}. We can argue as in the previous
case defining\vspace*{1pt} now $Y$ as a sum of $m$ independent copies of
$K^{1/2} G$.
Condition $ P( 4 \delta, \mathcal S, \mathbb C^{-1/2}
\widetilde Y
)\ge p/4$ is satisfied by \eqref{eq33ss}, since now $\mathcal L(
\widetilde Y)=\mathcal L( c_0 G )$. 

Following BG (\citeyear{BenGot97N1}), introduce the upper bound
$\varkappa(
t; N, X )$ for the characteristic function of quadratic forms; {cf.
} \citet{Ben84} and Bentkus, G\" otze and Zitikis (\citeyear
{BenGotZit93}).
We define
$\varkappa( t; N, X ) =\varkappa^* ( t; N, X )
+ \varkappa^* ( t; N,G )$, where
%
%
\begin{equation}
\varkappa^* ( t; N, X) = \sup_{x\in\mathbb R^d } \bigl| \E\operatorname{e}
\bigl\{ t \mathbb Q [Z_j ]+ \langle x, Z_j \rangle
\bigr\} \bigr|,\qquad  Z_j= {X}_{1}+\cdots+ {X}_{j},\hspace*{-35pt}
\label{eq323}
\end{equation}
with $ j=
\lfloor(N-2)/14 \rfloor$. Note that $ |\E
\operatorname
{e} \{ t \mathbb Q [Z_j ]+ \langle x, Z_j \rangle
\} |= |\E\operatorname{e} \{ t \mathbb Q [Z_j
-y] \} |$ with $y=-\mathbb Q x/(2t)$. In the sequel, we
use that
%
%
\begin{equation}
\varkappa\bigl(t; N, X' \bigr)\leq\varkappa\bigl(t; N,
X^\bullet\bigr).\label{eq1app}
\end{equation}
For the proof, it suffices to note
that $X'=X^\bullet-\E X^\bullet+W$ and $W$ is independent of $
X^\bullet$.

%
\begin{lemma}\label{L32}
Let the conditions of Lemma $\ref{L31}$ be satisfied. Then
%
%
\begin{eqnarray}
\label{eee}&&\int_{|t|\leq t_1 }\bigl( |t| K \bigr)^{\alpha}
\varkappa\bigl(t K; N, X^\bullet\bigr) \frac{ dt }{ |t| }
\nonumber
\\[-8pt]
\\[-8pt]
\nonumber
& &\qquad\ll_{\a,d} (\det\mathbb C)^{-1/2} \cases{ (Np )
^{-\a}, \qquad \mbox{for } 0\leq\a< d/2, \vspace*{2pt}
\cr
(N p)
^{-\a} \bigl(1+ \bigl\llvert\log(N p/m)\bigr\rrvert\bigr),\vspace*{2pt}\cr
\hspace*{63pt}\mbox{for } \a= d/2,\vspace*{2pt}
\cr
(N p) ^{-\a} \bigl(1+(N
p/m)^{(2 \a- d)/d} \bigr), \vspace*{2pt}\cr
\hspace*{63pt}\mbox{for } \a> d/2. }
\end{eqnarray}
\end{lemma}

Lemma \ref{L32} is a generalization of Lemma 3.2 from BG (\citeyear
{BenGot97N1})
which contains the same bound for $0\leq\a< d/2$. In this paper,
we have to estimate the left-hand side of \eqref{eee} in the case
$d/2\leq\a$ too.

\begin{pf}
We assume again that $ (p N )^{-1} m\le c_3(d)$
with sufficiently small $c_3(d)$ since otherwise \eqref{eee} is an
easy consequence of $\llvert\varkappa\rrvert\le1$.

By \eqref{eq33ss} and \eqref{eq321}, the condition
$ P( 4 \delta, \mathcal S_o, K^{1/2} \mathbb C^{-1/2} \widetilde
{Z}_m^{(\bullet)} )\ge p/4$ is fulfilled. Therefore, collecting
independent copies of $K^{1/2} X^\bullet$
in groups as in \eqref{eq320}, we can
apply inequality \eqref{equ71p} of Lemma \ref{GZ}.
By \eqref{MTq}, \eqref{eq33ss} and \eqref{equ71p}, for any $\g>0$
and $|t|\leq t_1$,
\[
\varkappa^* \bigl(t K; N, X^\bullet\bigr) \ll_{\g, d} (p
N/m)^{-\g}+ \min\bigl\{ 1; (Np/m)^{-d/2} \llvert t\rrvert
^{-d/2} (\det\mathbb C)^{-
1/2} \bigr\}.
\]
We have used that $\sigma^2=1$ implies
$\sigma_1^2\asymp_d1$. A similar upper bound is valid for the
quantity $\varkappa^* (t K; N, G)$; cf. the proof of
\eqref{eq317} for $k>N/2$. Thus we get for any $\g>0$ and
$|t|\leq t_1$,
\[
\varkappa\bigl(t K; N, X^\bullet\bigr)\ll_{\g, d} (p
N/m)^{-\g}+ \min\bigl\{ 1; (\det\mathbb C)^{-1/2} \bigl(m/(|t|
p N) \bigr)^{d/2} \bigr\}.
\]
Integrating this bound
(cf. the estimation of $I_1$ in Lemma \ref{L23}), we obtain
\eqref{eee}.
\end{pf}

\section{\texorpdfstring{Proof of Theorem \protect\ref{T15}}
{Proof of Theorem 2.2}}\label{s4}

To simplify notation, in Section \ref{s4} we write
$\Pi=\Pi_2^{\bullet}$ and $\L=\L_4^{\bullet}$. The assumption
$\s^2=1$ and equalities $\E\| \mathbb C^{-1/2} X \|^2=d$,
\eqref{eq14t} and~\eqref{eq15t} imply
%
%
\begin{equation}
\Pi+\L N\gg1,\qquad \Pi+\L\leq1,\qquad \s_j^2 \leq1,\qquad \det
\mathbb C\le1. \label{eq34}
\end{equation}

Recall that $\Delta_N^{(a)}$ and functions $\Psi_b$, $\Phi_b$ and
$\Theta_b$ are defined in \eqref{eq121} and
\eqref{eq118}--\eqref{edg}. Note now that
$\Theta_b^{\bullet}(x)=E_a^{\bullet}(x/N)$ and, according to
\eqref{edg},
%
%
\begin{equation}
\D_N^{(a)}\le\D_{N,\bullet}^{(a)}+\sup
_{x\in
\mathbb R} \bigl|\Theta_b (x)-\Theta_b^{\bullet}
(x) \bigr|, \label{eq310w}
\end{equation}
where $b=\sqrt N a$ and
%
%
\begin{equation}
\D_{N,\bullet}^{(a)}= \sup_{x\in\mathbb R }\bigl |
\Psi_b (x) - \Phi_b(x)-\Theta_b^{\bullet}(x)
\bigr|. \label{eq310ww}
\end{equation}

Let us verify that
%
%
\begin{equation}
\sup_{x\in\mathbb R }\bigl |\Theta_b (x)- \Theta_b^\bullet(x)
\bigr|\ll_d \Pi_3^\bullet. \label{eq312}
\end{equation}
To this end we apply representation
\eqref{eq121}--\eqref{eq122}
of the Edgeworth correction as a signed measure and estimate the
variation of that measure. Indeed, using
\eqref{eq121}--\eqref{eq122}, we have
%
%
\begin{eqnarray}\label{eq1}
&\displaystyle\sup_{x\in\mathbb R
} \bigl|\Theta_b (x)- \Theta_b^\bullet(x)
\bigr|\ll N^{-1/2} I,&
\nonumber
\\[-8pt]
\\[-8pt]
\nonumber
&\displaystyle I\= \int_{\mathbb R^d} \bigl|\E
p'''(x) X^3- \E
p'''(x) {X^\bullet}^3
\bigr| \,dx.&
\end{eqnarray}
By the explicit formula
\eqref{eq123}, the function $u \na p'''(x) u^3$ is a $3$-linear
form in the variable $u$. Therefore, using
$X= X^\bullet+X_{\bullet} $ and $ \|X^\bullet\| \|X_{\bullet}\|=0 $,
we have $ p'''(x) X^3- p'''(x) {X^\bullet}^3 = p'''(x)
{X_\bullet^3} $, and
%
%
\begin{equation}
N^{-1/2} I\leq3d^{3/2} \Pi_3^\bullet
\int_{\mathbb R^d} \bigl( \bigl\| \mathbb C^{-1/2} x\bigr \| + \bigl\|
\mathbb C^{-1/2} x \bigr\|^3 \bigr) p(x) \,dx = c_d
\Pi_3^\bullet.\label{eq222}
\end{equation}
Inequalities
\eqref{eq1} and
\eqref{eq222} imply now \eqref{eq312}.

To prove the statement of Theorem $\ref{T15}$, we have to derive
that
%
%
\begin{equation}
\D_{N,\bullet}^{(a)}\ll_d ( \Pi+ \L) \bigl(1+\Vert a
\Vert\bigr)^3 (\det\mathbb C)^{-1/2}. \label{eq310}
\end{equation}
While proving \eqref{eq310} we assume that
%
%
\begin{equation}
\Pi\leq c_d \quad \mbox{and} \quad\L\leq c_d,
\label{eq314}
\end{equation}
with a
sufficiently small positive constant $c_d$ depending on $d$ only.
These assumptions do not restrict generality.
Indeed, we
have $ |\Psi_b (x) - \Phi_b(x) | \leq1$. If
conditions~\eqref{eq314} do not hold, then the estimate
%
%
\begin{equation}
\sup_{x\in\mathbb R } \bigl|\Theta_b^\bullet(x) \bigr|
\ll_d N^{-1/2} \E\bigl\|\mathbb C^{-1/2}
X^\bullet\bigr\|^3\ll_d \L^{1/2}
\label{eq315}
\end{equation}
immediately implies \eqref{eq310}.
In order to prove \eqref{eq315} we can use \eqref{eq15t} and representation
\eqref{eq121}--\eqref{eq122} of the Edgeworth correction.
Estimating the variation of that measure and using
%
%
\begin{eqnarray}
\E\bigl\|\mathbb C^{-1/2} X^\bullet\bigr\|^2&\leq&\E\bigl\|\mathbb
C^{-1/2} X \bigr\|^2=d,\label{eq310u}
\\
\bigl(\E\bigl\|\mathbb C^{-1/2} X^\bullet\bigr\|^3
\bigr)^2 &\leq&\E\bigl\|\mathbb C^{-1/2} X^\bullet
\bigr\|^2 \E\bigl\|\mathbb C^{-1/2} X^\bullet\bigr\|^4,
\label{eq310uu}
\end{eqnarray}
we obtain \eqref{eq315}.

It is clear that
%
%
\begin{eqnarray}\label{eq310x}
&&\D_{N,\bullet}^{(a)}\le\sup_{x\in\mathbb R} \bigl(\bigl |
\Psi_b (x)-\Psi_b' (x)\bigr |+\bigl |\Theta
_b^\bullet(x)-\Theta_b' (x)\bigr |
\nonumber
\\[-8pt]
\\[-8pt]
\nonumber
&&\hspace*{89pt}{}+ \bigl|\Psi_b' (x) - \Phi_b(x)-
\Theta'_b(x) \bigr| \bigr).
\end{eqnarray}
Similarly to \eqref{eq1}, we have
%
%
\begin{eqnarray}\label{eq2}
&\displaystyle\sup_{x\in\mathbb R
} \bigl|\Theta_b^\bullet(x)-
\Theta_b' (x) \bigr|\ll N^{- 1/2} J,&
\nonumber
\\[-8pt]
\\[-8pt]
\nonumber
 & \displaystyle J\= \int
_{\mathbb R^d} \bigl|\E p'''(x)
{X^\bullet}^3- \E p'''(x)
{X'}^3 \bigr| \,dx.&
\end{eqnarray}
Recall that vector $X'$ is
defined in \eqref{eq120}.
By Lemma \ref{L24}, we have\break
$\E\|\mathbb C^{-1/2} W \|^2\leq2d \Pi$ (hence, $\E
\|\mathbb C^{-1/2} W \|^q\ll_d \Pi^{q/2}$, for $0\le q\le2$).
Using the well-known equivalence of moments of Gaussian random
vectors, we conclude that
%
%
\begin{equation}
\E\bigl\|\mathbb C^{-1/2} W \bigr\|^q\ll_q \bigl(\E\bigl\|
\mathbb C^{-1/2} W \bigr\|^2 \bigr)^{q/2}
\ll_{q, d} \Pi^{q/2},\qquad  q\ge0.\label{eq442}
\end{equation}
Furthermore, according to
\eqref{eq14t}, \eqref{eq15t} and \eqref{eq314},
%
%
\begin{equation}
\label{eq234}\E\bigl\|\mathbb C^{-1/2} X_\bullet\bigr\| \ll_d
\Pi N^{-
1/2}\ll_d\Pi^{1/2}N^{-1/2}.
\end{equation}
Hence, by \eqref{eq15t}, \eqref{eq120}, \eqref{eq34}, \eqref
{eq442} and
\eqref{eq234},
%
%
\begin{equation}
\label{eq2345} \E\bigl\|X' \bigr\|^4\ll\ovln\beta\=\E\bigl\| \mathbb
C^{-
1/2} X' \bigr\|^4\ll_d N\L+
\Pi^2.
\end{equation}
Using \eqref{eq123}, \eqref{eq34}, \eqref{eq314},
\eqref{eq310u} and \eqref{eq2}--\eqref{eq234}, we get
%
%
\begin{eqnarray}\label{eq223}\qquad
N^{-1/2} J&\ll_d &\Pi^{1/2}\bigl(N^{-1/2}
\Pi+\Lambda^{1/2}\bigr) \int_{\mathbb R^d} \bigl( \bigl\|
\mathbb C^{-1/2} x \bigr\| + \bigl\| \mathbb C^{-1/2} x\bigr \|^3
\bigr) p(x)\,dx
\nonumber
\\[-8pt]
\\[-8pt]
\nonumber
&\ll_d& \Pi+\Lambda.
\end{eqnarray}
Thus, according to
\eqref{eq2} and \eqref{eq223},
%
%
\begin{equation}
\sup_{x\in\mathbb R
} \bigl|\Theta_b^\bullet(x)-
\Theta_b' (x) \bigr|\ll_d \Pi+\Lambda.
\label{eqpo}
\end{equation}

The same approach is applicable for the estimation of
$ |\Theta_b' |$. Using \eqref{eq121}--\eqref{eq123},
\eqref{eq120}, \eqref{eq34}, \eqref{eq310u}, \eqref{eq310uu},
\eqref{eq442} and \eqref{eq234}, we get
%
%
\begin{eqnarray}
\sup_{x\in\mathbb R
} \bigl|\Theta_b' (x)\bigr |&\ll&
N^{-1/2} \int_{\mathbb R^d} \bigl| \E p'''(x)
{X'}^3\bigr | \,dx
\nonumber
\\[-8pt]
\\[-8pt]
\nonumber
&\ll_d&\Lambda^{1/2} +N^{-1/2}\Pi^{3/2}.
\label{eq999}
\end{eqnarray}

Let us prove that
%
%
\begin{equation}
\sup_{x\in\mathbb R } \bigl|\Psi_b (x) - \Psi'_b
(x) \bigr|\ll(\det\mathbb C)^{-1/2} p^{-2} (\Pi+\L) \bigl(1+\Vert
a\Vert^2\bigr). \label{eq325}
\end{equation}
Using truncation [see
\eqref{eq28}], we have
$ | \Psi_b - \Psi_b^\bullet|\leq\Pi$,
and
%
%
\begin{equation}
\sup_{x\in\mathbb R } \bigl|\Psi_b (x) - \Psi'_b
(x) \bigr| \leq\Pi+ \sup_{x\in\mathbb R } \bigl| \Psi_b^\bullet(
x ) - \Psi'_b (x) \bigr|. \label{eq326}
\end{equation}
In order to estimate $| \Psi_b^\bullet- \Psi'_b |$, we
apply Lemmas \ref{L31} and \ref{L32}. The number $m$ in these
Lemmas exists and $N \L/p\gg_d 1$, as it follows from
\eqref{eq34} and \eqref{eq314}. Let us choose the minimal
$m$, that is,
$m\asymp_d N \L/p $. Then $(p N )^{-1} m \ll_d \L/p^2 $ and
$m/N\ll_d \L/p$. Therefore,
using
Lemma \ref{L31}, we have
%
%
\begin{eqnarray}\label{eq327}
&&\sup_x \bigl| \Psi_b^\bullet( x ) -
\Psi_b ' (x) \bigr|
\nonumber\hspace*{-35pt}
\\[-8pt]
\\[-8pt]
\nonumber
&&\qquad\ll_d p^{-2}
\L(\det\mathbb C)^{-1/2}+\int_{|t|\leq t_1}\bigl | \widehat
\Psi_b^\bullet(\tau)- \widehat\Psi'_b
(\tau)\bigr | \frac{ dt }{ |t| },\qquad \tau=t K.\hspace*{-35pt}
\end{eqnarray}

We shall prove that
%
%
\begin{equation}
\bigl|\widehat\Psi_b^\bullet(\tau)-\widehat
\Psi_b ' (\tau) \bigr| \ll_d \varkappa\Pi|\tau|
N \bigl(1+|\tau| N\bigr) \bigl(1+\Vert a\Vert^2\bigr) \label{eq328}
\end{equation}
with
$\varkappa= \varkappa(\tau; N, X^\bullet)$. Combining
\eqref{eq326}--\eqref{eq328}, using $\tau= t K$ and
integrating
inequality \eqref{eq328}
with the help of Lemma \ref{L32},
we derive \eqref{eq325}.

Let us prove \eqref{eq328}.
Writing $ D= Z_N^{(\bullet)} -\E Z_N^{(\bullet)}-b $,
we have
\[
Z_N^{(\bullet)}-b = D+\E Z_N^{(\bullet)},\qquad
\mathcal L\bigl(Z_N'-b\bigr) = \mathcal L(D+ \sqrt{N}
W)
\]
and
%
%
\begin{equation}
\label{eq327m} \bigl|\widehat\Psi_b^\bullet(\tau)-\widehat
\Psi'_b(\tau) \bigr| \leq\bigl| f_1(\tau) \bigr|+ \bigl|
f_2(\tau) \bigr|
\end{equation}
with
%
%
\begin{eqnarray}\label{eq329}
 f_1(\tau)&=& \E\operatorname{e} \bigl\{ \tau \mathbb Q [
D +\sqrt{N} W ] \bigr\}- \E\operatorname{e} \bigl\{ \tau \mathbb
Q [ D ]
\bigr\},
\nonumber
\\[-8pt]
\\[-8pt]
\nonumber
f_2(\tau)&=& \E\operatorname{e} \bigl\{ \tau \mathbb Q \bigl[ D +
\E Z_N^{(\bullet)} \bigr] \bigr\}- \E\operatorname{e} \bigl\{
\tau
\mathbb Q [ D] \bigr\}.
\end{eqnarray}
Now we have to prove that
both $ | f_1(\tau) |$ and $ | f_2(\tau) |$ may be
estimated by the right-hand side of \eqref{eq328}.

Let us consider $f_1$. We can write $\mathbb Q [ D +\sqrt{N}
W ]= \mathbb Q [ D ]+A+B$ with $A=2 \sqrt{N} \langle\mathbb
Q D, W\rangle$ and $B=N \mathbb Q [ W ]$. Taylor's
expansions of the exponent in \eqref{eq329} in powers of $i\tau
B$ and $i \tau A$ with remainders $\O( \tau B)$ and
$\O( \tau^2 A^2)$, respectively, imply (recall that $\E W=0$ and
$\mathbb Q^2=\mathbb I_d$)
%
%
\begin{equation}
\bigl| f_1(\tau) \bigr|\ll\varkappa |\tau| N \E\|W\|^2 +
\varkappa \tau^2 N \E\|W\|^2 \E\|D\|^2,
\label{eq330}
\end{equation}
where $\varkappa= \varkappa
(\tau; N, X^\bullet)$. The estimation of the remainders of these
expansions is based on the splitting
and conditioning techniques described in Section 9 of BG (\citeyear
{BenGot97N1});
see also Bentkus, G\" otze and Zaitsev (\citeyear{BenGotZai97}). Using
the relations
$\E\|W\|^2 \ll\E\|\mathbb C^{-1/2}W\|^2 \ll_d \Pi$, $\s^2=1$
and $\E\|D\|^2 \ll N(1+\Vert a\Vert^2)$, we derive from
\eqref{eq330} that
%
%
\begin{equation}
\bigl| f_1(\tau) \bigr| \ll_d \varkappa \Pi |\tau| N \bigl( 1
+ |\tau| N  \bigr) \bigl(1+\Vert a\Vert^2\bigr).\label{eq331}
\end{equation}
Note that $\E Z_N^{(\bullet)} = N \E
X^\bullet= -N \E X_\bullet$. Expanding the exponent $\operatorname
{e} \{
\tau \mathbb Q [ D + \E Z_N^{(\bullet)} ] \}$, using
\eqref{eq234} and proceeding similarly to the proof
of \eqref{eq331}, we obtain
%
%
\begin{equation}
\bigl| f_2(\tau) \bigr|\ll_d \varkappa \Pi|\tau| N\bigl(1+\Vert
a\Vert\bigr).\label{eq331f}
\end{equation}
Inequalities \eqref{eq327m}, \eqref{eq331}
and \eqref{eq331f} imply now \eqref{eq328}.

It remains to estimate $ |\Psi'_b - \Phi_b - \Theta_b' |$.
Recall that the distribution functions $\Psi_b^{(l)}(x)$, for $0\leq
l\leq N$,
are defined in \eqref{eq316}.

Fix an integer $k$, $1\leq k\leq N$. Clearly, we have
%
%
\begin{equation}
\label{eq156}\sup_{x\in\mathbb R } \bigl|\Psi'_b
(x) - \Phi_b(x) - \Theta_b' (x)\bigr |\le
I_1+I_2+I_3,
\end{equation}
where
%
%
\begin{eqnarray}
\label{eq156b}I_1&=&\sup_{x\in\mathbb R } \bigl|\Psi
_b^{(k)} (x) - \Phi_b(x)- (N-k)
\Theta_b' (x)/N \bigr|,
\\
\label{eq156a}I_2&=&\sup_{x\in\mathbb R } \bigl|
\Psi'_b (x) - \Psi_b^{(k)}(x)\bigr |
\end{eqnarray}
and
%
%
\begin{equation}
\label{eq156c}I_3=\sup_{x\in\mathbb R } kN^{-1} \bigl|
\Theta_b' (x) \bigr|.
\end{equation}

Let estimate $I_1$.
Define the distributions
%
%
\begin{eqnarray}
\label{eq123aq}\mu(A)&=& \P\Biggl\{ U_k +\sum
_{j=k+1}^N X_j' \in\sqrt{N}
A \Biggr\},
\nonumber
\\[-8pt]
\\[-8pt]
\nonumber
 \mu_0(A)&=& \P\{ U_N \in\sqrt{N}
A \}= \P\{ G \in A \},
\end{eqnarray}
where $U_l
=G_1+\cdots+G_l $. Introduce the measure $\chi'$ replacing $X$ by
$X'$ in \eqref{eq122}. For the Borel sets $A\subset\mathbb R^d$ define
the Edgeworth correction (to the distribution $\mu$) as
%
%
\begin{equation}
\mu_1^{(k)} (A)= (N-k) N^{-3/2}\chi'
(A)/6.
\end{equation}
Introduce
the signed measure
%
%
\begin{equation}
\nu=\mu-\mu_0-\mu_1^{(k)}.\label{eq122aq}
\end{equation}

It is easy to see that a re-normalization of random vectors
implies [see relations~\eqref{eq121},
{\eqref{eq118}--\eqref{edg}}, \eqref{eq316} and
\eqref{eq123aq}--\eqref{eq122aq}]
%
%
\begin{eqnarray}\label{eq156d}
\bigl|\Psi_b^{(k)} (x) - \Phi_b(x)- (N-k)
\Theta_b' (x)/N \bigr| &=&\nu\bigl( \bigl\{u\in\mathbb
R^d\dvtx\mathbb Q[u-a]\le x/N \bigr\} \bigr)
\nonumber\hspace*{-35pt}
\\[-8pt]
\\[-8pt]
\nonumber
&\le&\delta_N\= \sup_{A \subset\mathbb R^d } \bigl| \nu(A) \bigr|.\hspace*{-35pt}
\end{eqnarray}

%
\begin{lemma}\label{L94} Assume that $ d<\infty$
and $1\leq k\leq N$. Then there exists a $c(d)$ depending on $d$ only
and such that $\delta_N$
defined in $\eqref{eq156d}$ satisfies the inequality
%
%
\begin{equation}
\delta_N \ll_d \frac{ \ovln\b}{ N } + \frac{ N^{d/2} }{ k^{d/2} }
\exp\bigl\{ - c(d) k/\ovln\beta\bigr\} \label{eq914}
\end{equation}
with $\ovln\b=\E\|\mathbb C^{-1/2}X'\|^4$.
\end{lemma}

\textit{An outline of the proof}. We repeat and slightly improve the
proof of Lemma 9.4 in BG (\citeyear{BenGot97N1}); cf. the proof of
Lemma 2.5 in BG
(\citeyear{BenGot97N1}). We shall prove \eqref{eq914} assuming that
$\cov X= \cov
X'=\cov G=\mathbb I_d$. Applying it to $\mathbb C^{-1/2} X'$ and
$\mathbb C^{-1/2} G$, we obtain \eqref{eq914} in general case.\vadjust{\goodbreak}

While proving \eqref{eq914} we assume
that $\ovln\b/N \leq c_d$ and $N \geq1/ c_d$ with a sufficiently
small positive constant $c_d$. Otherwise \eqref{eq914} follows
from the obvious bounds $\ovln\b\ge\s^4=d^2$ and
\[
\delta_N\ll_d 1 + (\ovln\b/N)^{1/2} \int
_{\mathbb R^d
} \|x\|^3 p(x) \,dx \ll_d 1 + (
\ovln\b/N)^{1/2}.
\]

Set $n=N-k$.
Denoting by $Z_j^\prime$ and $U_j^\prime$ sums of $j$ independent
copies of $X^\prime$ and~$G^\prime$, respectively, introduce the
multidimensional characteristic functions
%
%
\begin{eqnarray}
g(t)&=&\E\operatorname{e} \bigl\{ \bigl\langle N^{-1/2} t, G\bigr
\rangle
\bigr\},\qquad h(t)=\E\operatorname{e} \bigl\{ \bigl\langle N^{-1/2} t,
X'\bigr\rangle\bigr\},\label{eqe1}
\\
\label{eq917a}f(t)&=& \E\operatorname{e} \bigl\{ \bigl\langle
N^{-1/2}
t, Z_{n}^\prime\bigr\rangle\bigr\}=h^{n}(t),
\nonumber
\\[-8pt]
\\[-8pt]
\nonumber
f_0(t)&=& \E\operatorname{e} \bigl\{ \bigl\langle N^{-1/2}
t, U_{n}^\prime\bigr\rangle\bigr\}=g^{n}(t),
\\
f_1(t)&=& n m(t) f_0(t)\qquad \mbox{where } m(t)=
\frac
{ 1 }{ 6 {N^{3/2}} } \E\bigl\langle i t, X^\prime\bigr
\rangle^3,
\\
\widehat\nu(t)&=&\bigl(f(t)-f_0(t)-f_1(t)\bigr) g(\rho
t), \qquad\rho^2 =k.\label{eq916s}
\end{eqnarray}
It is easy to see that
%
%
\begin{equation}
\label{Fur1}\widehat\nu(t)=\int_{\mathbb R^d} \operatorname
{e}\bigl
\{\langle t, x\rangle\bigr\} \nu(dx).
\end{equation}
Using a truncation, we obtain
%
%
\begin{equation}
\E\bigl\| Z_l^\prime/\sqrt N \bigr\|^\gamma
\ll_{\gamma,d} 1,\qquad  \gamma>0, 1\le l\le N. \label{eq915}
\end{equation}

By an extension of the proof of Lemma 11.6 in Bhattacharya and Rao
(\citeyear{BhaRan86}) [see also the proof of Lemma 2.5 in BG
(\citeyear{BenGot96})], we obtain
%
%
\begin{equation}
\delta_N\ll_d \max_{|\a| \leq2d} \int
_{\mathbb R^d} \bigl|\partial^\a\widehat\nu(t) \bigr| \,dt.
\label{eq916}
\end{equation}
Here
$|\a|=|\a_1 |+\cdots+|\a_d |$, $\a=(\a_1,\ldots,\a_d)$,
$\a_j\in\mathbb Z$, $\a_j\ge0$. In order to derive~\eqref{eq914}
from \eqref{eq916}, it suffices to prove that, for $|\a|\leq
2d$,
%
%
\begin{eqnarray} \quad 
\bigl|\partial^\a\widehat\nu(t)\bigr |&\ll_d& g( c_1
\rho t), \label{eq917}
\\
\qquad\bigl|\partial^\a\widehat\nu(t) \bigr| &\ll_d& \ovln\b
N^{-1} \bigl(1+\|t\|^{6} \bigr) \exp\bigl\{ -
c_2 \|t\|^2\bigr\}\qquad \mbox{for } \|t\|^2
\leq c_3(d) N/\ovln\b.\label{eq918}
\end{eqnarray}
Indeed, using \eqref{eq917} and denoting
$T= \sqrt{c_3(d) N/\ovln\b}$, we obtain
%
%
\begin{eqnarray}\label{eq919}
\qquad\int_{\|t\|\geq T} \bigl|\partial^\a\widehat\nu(t) \bigr| \,dt
&\ll_d& \int_{\|t\|\geq T} g( c_1\rho t) \,dt
\nonumber
\\[-8pt]
\\[-8pt]
\nonumber
&
\ll_d& \frac{ N^{d/2} }{ \rho^{ d} } \exp\biggl\{ - \frac{ c_1^2
\rho^{2} T^2 }{ 8N } \biggr\}
\int_{\mathbb R^d} \exp\bigl\{ - c_1^2 \|t
\|^2/8\bigr\} \,dt,
\end{eqnarray}
and it is easy to see that
the right-hand side of \eqref{eq919} is bounded from above by the
second summand on the right-hand side of \eqref{eq914}.
Similarly, using \eqref{eq918}, we can integrate
$ |\partial^\a\widehat\nu(t) |$ over $\|t\|\leq T$, and the
integral is bounded from above by $c_d \ovln\b/N$.

In the proof of \eqref{eq917}--\eqref{eq919} we applied standard
methods of estimation which are provided in Bhattacharya and Rao
(\citeyear{BhaRan86}). In particular, we used a Bergstr\" om type identity
%
%
\begin{equation}\qquad
f-f_0-f_1=\sum_{j=0}^{n-1}(h-g-m)
h^j g^{n-j-1} +\sum_{j=0}^{n-1}
m\sum_{l=0}^{j-1}(h-g) h^l
g^{n-l-1}, \label{eqe2}
\end{equation}
relations \eqref{eqe1}--\eqref{eq915}, $1\leq
k\leq N$, $ |\partial^\a\exp\{ -c_4 \|t\|^2\} |\ll_\a
\exp\{ -c_5 \|t\|^2\}$,\break  ${\sqrt{N}/\ovln\b^{1/2}\gg_d1}$ and
$y^{c_d} \exp\{ -y\} \ll_d 1 $, for $y>0$. 

Applying \eqref{eq156b}, \eqref{eq156d}
and Lemma \ref{L94}, we get
%
%
\begin{equation}
I_1\ll_d \frac{ \ovln\b}{ N } + \frac{ N^{d/2} }{ k^{d/2} } \exp
\bigl\{ - c(d) k/\ovln\b\bigr\}. \label{eq340}
\end{equation}

For the estimation of $I_2$ we shall use Lemma \ref{L93} which is
an easy consequence of BG [(\citeyear{BenGot97N1}), Lemma 9.3], \eqref
{eq1app} and
\eqref{eq2345}.

%
\begin{lemma}\label{L93}
We have
\[
\bigl|\widehat\Psi_b' (t)-\widehat\Psi_b^{(l)}(t)
\bigr| \ll\varkappa t^2 l \bigl( \ovln\b+|t| N \ovln\b+|t| N
\sqrt{N \ovln\b} \bigr) \bigl(1+\Vert a\Vert^3\bigr)\qquad \mbox{for }0\leq
l\leq N,
\]
where $\varkappa= \varkappa( t; N, X^\bullet)$; cf. \eqref{eq323}.
\end{lemma}

As in the proof of \eqref{eq327},
applying Lemma \ref{L31} [choosing $m\asymp_d N (\L+\Pi) /p $]
and using \eqref{eq33ss},
we obtain
\[
I_2\ll_d (\L+\Pi) (\det\mathbb C)^{-1/2} + \int
_{|t|\leq t_1
} \bigl|\widehat\Psi'_b (\tau) -
\widehat\Psi_b^{(k)} (\tau)\bigr | \,dt /|t|,\qquad \tau=t K.
\]
The existence of such an $m$ is ensured
by \eqref{eq33ss}, \eqref{eq34} and \eqref{eq314}. Applying
Lemma~\ref{L93} and replacing in that lemma $t$ by $\tau$, we
have
%
%
\begin{equation}
\label{eq48} \bigl|\widehat\Psi_b '(\tau) - \widehat
\Psi_b^{(k)}(\tau)\bigr | \ll\varkappa \tau^2 k
\bigl( \ovln\b+| \tau| N \ovln\b+| \tau| N \sqrt{N \ovln\b} \bigr) \bigl
(1+\Vert
a\Vert^3\bigr).
\end{equation}
Integrating with the help of Lemma \ref{L32} and using \eqref{eq33ss},
we obtain
%
%
\begin{eqnarray}\label{eq339}
I_2&\ll_d &(\det\mathbb C)^{-1/2} \bigl(\Pi+\L+ k
N^{-2} ( \ovln\b+ \sqrt{N \ovln\b} )
\nonumber
\\[-8pt]
\\[-8pt]
\nonumber
&&{}\times
 \bigl( 1 + (\Pi+
\L)^{-1/d} \bigr) \bigl(1+\Vert a\Vert^3\bigr) \bigr).
\end{eqnarray}

Let us choose
$k\asymp_d N^{1/4} {\ovln\b}^{3/4} $.
Such $k\leq N$ exists by $\ovln\b\gg_d\s^4=1$, by \eqref{eq2345}
and by
assumption \eqref{eq314}.
Then \eqref{eq340} and \eqref{eq339} turn
into
%
%
\begin{equation}
I_1\ll_d \frac{ \ovln\b}{ N } + \biggl( \frac{ N }{ {\ovln\b} }
\biggr)^{3d/8} \exp\biggl\{ -c_d \biggl( \frac{ N }{ {\ovln\b} }
\biggr)^{1/4} \biggr\} \ll_d \frac{ \ovln\b}{ N }
\label{eq4899}
\end{equation}
and
%
%
\begin{eqnarray}\label{eq3399}
I_2&\ll_d &(\det\mathbb C)^{-1/2} \biggl(\Pi+\L+
\biggl( \biggl( \frac{ \ovln\b}{ N } \biggr)^{5/4}+ \biggl(
\frac{
\ovln\b}{ N } \biggr)^{7/4} \biggr)
\nonumber
\\[-8pt]
\\[-8pt]
\nonumber
&&{}\times \bigl( 1 + (\Pi+
\L)^{-1/d} \bigr) \bigl(1+\Vert a\Vert^3\bigr) \biggr).
\end{eqnarray}
Using \eqref{eq33ss},
\eqref{eq314},
\eqref{eq2345} and \eqref{eq3399}, we get
%
%
\begin{equation}
I_2\ll_d (\det\mathbb C)^{-1/2} \biggl(\Pi+\L+
\frac{ \ovln\b}{ N } \bigl(1+\Vert a\Vert^3\bigr) \biggr).
\label{eq33999}
\end{equation}

Finally, by \eqref{eq314}, \eqref{eq2345}, \eqref{eq999} and \eqref
{eq156c},
%
%
\begin{equation}
I_3\ll_d \frac{ k }{ N } \bigl(\Lambda^{1/2}
+N^{-1/2}\Pi^{3/2} \bigr) \ll\Lambda+\Pi.\label{eq777}
\end{equation}

Inequalities \eqref{eq314}, \eqref{eq310x}, \eqref{eq2345}, \eqref{eqpo},
\eqref{eq325}, \eqref{eq156},
\eqref{eq4899}, \eqref{eq33999} and \eqref{eq777} imply now \eqref{eq310}
[and, hence, \eqref{eq18w}] by an application of
$\Pi+\L\leq1$. Note that, by \eqref{eq15t}, we have $\Pi\leq\Pi
_3^\bullet$.
Together with
\eqref{eq310w} and \eqref{eq312}, inequality \eqref{eq310}
yields~\eqref{eq18}. The statement of Theorem \ref{T15}
is proved.
\end{pf}

%

\section{From probability to number theory}
\label{s5}

In Section \ref{s5} we reduce the estimation of the integrals of
the modulus of characteristic functions $\widehat\Psi_b(t)$ to the
estimation the integrals of some theta-series. We shall use the
following lemmas.

%
\begin{lemma}[{[BG (\citeyear{BenGot97N1}), Lemma 5.1]}]\label{L51}
Let $L,C\in
\mathbb
R^d$ and let $\mathbb Q\dvtx\mathbb R^d\to\mathbb R^d$ be a symmetric
linear operator.
Let $Z,U,V$ and $ W$ denote independent random vectors taking
values in $\mathbb R^d$. Denote by
\[
P(x) = \langle\mathbb Q x,x\rangle+ \langle L,x\rangle+C,\qquad x \in
\mathbb
R^d,
\]
a real-valued polynomial of second order. Then
\[
2 \bigl|\E\operatorname{e} \bigl\{ t P(Z+U+V+W) \bigr\} \bigr|^2 \leq\E
\operatorname{e} \bigl\{ 2 t \langle\mathbb Q \widetilde
Z,\widetilde U\rangle
\bigr\} + \E\operatorname{e} \bigl\{ 2 t \langle\mathbb Q
\widetilde Z,
\widetilde V\rangle\bigr\}.
\]
\end{lemma}

Let $\delta>0$, $\mathcal S = \{ e_{1}, \ldots, e_{s} \} \subset
\mathbb R^d$ and let $\mathbb D\dvtx\mathbb R^d\to\mathbb R^d$ be a linear
operator.
Usually, we take $\mathbb D=\mathbb C^{-1/2}$. Denote
%
%
\begin{equation}
\label{eq7yy} \qquad{\bolds{\Gamma}} (\delta;\mathbb D, \mathcal S)=
\bigl\{(z_1,\ldots,z_s)\dvtx z_j\in\mathbb
R^d, \|\mathbb Dz_j -e_j\|\leq\delta,
\mbox{ for all }1\leq j\leq s \bigr\}.
\end{equation}
Recall that $\mathcal S_o=\{e_{1}, \ldots, e_{s}\}\subset\mathbb R^d
$ denotes
an orthonormal system.

Let $ \{\varepsilon_{jk}, j=1, 2\ldots, s; k=1,
2\ldots\}\cup\{\varepsilon_{jk}', j=1, 2\ldots, s; k=1,
2\ldots\}$ be i.i.d. symmetric Rademacher random
variables.

%
\begin{lemma}\label{L63} Assume that $\mathbb Q^2= \mathbb I_d$
and that the condition $ P(\delta, \mathcal S,\mathbb D
\widetilde X
)\ge p$ holds with some $p>0$ and $\delta>0$. Write $m = \lfloor
{p N }/ (5 s) \rfloor$.
Then, for
any $0<A\leq B$, $b\in\mathbb R^d$ and $\gamma>0$, we have
%
%
\begin{equation}
\label{eq71qq}\int_A^B\bigl | \widehat
\Psi_b (t) \bigr| \frac{ dt }{ | t | } \leq I+ c_\g(s) (p
N)^{-\g} \log\frac{ B }{ A },
\end{equation}
with
%
%
\begin{equation}
\label{pppp} I= \sup_\Gamma\sup_{b\in\mathbb R^d} \int
_A^B \sqrt{\varphi_b ( t/4)}
\frac{ dt }{ | t | },\qquad \varphi_b (t) \= \bigl| \E\operatorname{e} \bigl
\{ t \mathbb Q [Y + b] \bigr\} \bigr|^2,
\end{equation}
where $Y= \sum_{k=1}^{m}U_{k}$
denote a sum of independent (non i.i.d.) vectors
$U_k=\sum_{j=1}^{s}\varepsilon_{jk}z_{jk}$, and
$\sup_\Gamma$ is taken over all
$ \{(z_{1k},\ldots,z_{sk}) \in
{\bolds{\Gamma}} (\delta;\mathbb D, \mathcal S ), k=\break 1,\ldots
, m \}$.
\end{lemma}

Lemma \ref{L63} is an analogue of Corollary 6.3 from BG (\citeyear
{BenGot97N1}).
Its proof is even simpler than that in BG (\citeyear{BenGot97N1}).
Therefore it is
omitted.

%
\begin{lemma}\label{L73}
Assume that $\mathbb Q^2= \mathbb I_d$
and that the condition $ P(\delta, \mathcal S,\mathbb D
\widetilde X
)\ge p$ holds with some $p>0$ and $\delta>0$. Let
%
%
\begin{equation}
\label{dfn}n \= \bigl\lfloor{p N}/({16 s}) \bigr\rfloor\ge1.
\end{equation}
Then, for
any $0<A\leq B$, $b\in\mathbb R^d$ and
$\gamma>0$,
%
%
\begin{equation}
\int_A^B\bigl | \widehat\Psi_b (t)
\bigr| \frac{
dt }{ | t | } \leq c_\g(s) (p N)^{-\g} \log
\frac{ B }{ A } + \sup_\Gamma\int_A^B
\sqrt{\E\operatorname{e} \bigl\{ t \bigl\langle\mathbb Q
\widetilde W,\widetilde
W' \bigr\rangle/2 \bigr\}} \frac{ dt }{ | t | },\label{eq71}\hspace*{-35pt}
\end{equation}
and
for any fixed $t\in\mathbb R$,
%
%
\begin{equation}
\bigl| \widehat\Psi_b (t)\bigr | \leq c_\g(s) (p
N)^{-\g}+ \sup_\Gamma\sqrt{\E\operatorname{e} \bigl\{
t \bigl\langle\mathbb Q \widetilde W,\widetilde W' \bigr\rangle/2
\bigr\}},\label{equ71}
\end{equation}
where
$W= {V}_{1}+\cdots+ {V}_{n} $ and $W'=V_1'+\cdots+ V_n' $ are
independent sums
of independent copies of random vectors
$V=\sum_{j=1}^{s}\varepsilon_{j1}z_{j}$ and
$V'=\sum_{j=1}^{s}\varepsilon'_{j1}z'_{j}$, and
$\sup_\Gamma$ is taken over all
$ (z_{1},\ldots,z_{s}), (z'_{1},\ldots,z'_{s}) \in
{\bolds{\Gamma}} (\delta;\mathbb D, \mathcal S )$.
\end{lemma}

Note that this lemma will be proved for
general $\mathcal S$, but in this paper we need $\mathcal S=\mathcal S_o$
only. Moreover, a more careful estimation of
binomial probabilities could allow us to replace $c_\g(s) (p N)^{-\g
}$ in
\eqref{eq71qq}, \eqref{eq71} and \eqref{equ71}
by
$c (s) \exp\{ -cp N \}$; see, for example, \citet{NagChe05}.
However, we do not need to use this improvement.

\begin{pf*}{Proof of Lemma \ref{L73}} Inequality \eqref{equ71} is
an analogue of the statement of Lemma 7.3 from BG (\citeyear
{BenGot97N1}). Its
proof is even simpler than that in BG (\citeyear{BenGot97N1}).
Therefore it is
omitted.

Let us show that
%
%
\begin{equation}
\int_A^B \bigl| \widehat\Psi_b (t)
\bigr| \frac{ dt }{ | t | } \leq c_\g(s) (p N)^{-\g} \log
\frac{ B }{ A } + \sup_\Gamma\int_A^B
\sqrt{\E\operatorname{e} \bigl\{ t \bigl\langle\mathbb Q
\widetilde W,\widetilde
W' \bigr\rangle/2 \bigr\}} \frac{ dt }{ | t | },\label{eq71q}\hspace*{-35pt}
\end{equation}
where
$W= V_1+\cdots+ V_n$ and $W'=V_1'+\cdots+ V_n' $ are independent
sums of independent (non i.i.d.) vectors $V_k=\sum
_{j=1}^{s}\varepsilon_{jk}z_{jk}$, and
$V'_k=\sum_{j=1}^{s}\varepsilon'_{jk}z'_{jk}$, respectively, while
$\sup_\Gamma$ is taken over all
$ \{(z_{1k},\ldots,z_{sk}), (z'_{1k},\ldots,z'_{sk}) \in
{\bolds{\Gamma}} (\delta;\break \mathbb D, \mathcal S ), k=1,\ldots
,n \}$.

Comparing \eqref{eq71} and \eqref{eq71q}, we see that inequality
\eqref{eq71q} is related to sums of \textit{non i.i.d.} vectors
$\{V_j\}$ and $\{V_j'\}$ while inequality \eqref{eq71} deals with
{i.i.d.} vectors. Nevertheless, we derive \eqref{eq71} from
\eqref{eq71q}.

While proving \eqref{eq71q} we can assume that $p N\geq c_s$
with a sufficiently large constant $c_s$, since otherwise \eqref{eq71q}
is obviously valid.

Let $\varphi_b(t)$ be defined in \eqref{pppp}, where $Y=
\sum_{k=1}^{m}U_{k}$ is a sum of independent (non i.i.d.)
vectors $U_k=\sum_{j=1}^{s}\varepsilon_{jk}z_{jk}$, where $
\{(z_{1k},\ldots,z_{sk}) \subset
{\bolds{\Gamma}} (\delta;\mathbb D, \mathcal S ), k=1,\ldots
,m \}$,
$m = \lfloor
{p N }/ (5 s) \rfloor$.

We shall
apply the symmetrization Lemma \ref{L51}.
Split
$ Y=T+T_1+T_2 $ into
sums of independent sums of independent summands
so that each of the sums $ T$, $T_{1}$ and $ T_{2}$ contains
$n=\lfloor p N /(16 s) \rfloor$
independent summands
$U_j$. Such an $n$ exists since $p N\geq c_s$
with a sufficiently large $c_s$.
Lemma \ref{L51} implies that
%
%
\begin{equation}
\label{434}2 \varphi_b (t) \leq\E\operatorname{e} \bigl\{ 2 t
\langle\mathbb Q \widetilde T,\widetilde T_{1} \rangle\bigr\} + \E
\operatorname{e} \bigl\{ 2 t \langle\mathbb Q \widetilde T,
\widetilde
T_{2} \rangle\bigr\}.
\end{equation}
Inequality \eqref{eq71q} follows now from \eqref{434} and
Lemma \ref{L63}.

Let now $W= V_1+\cdots+ V_n$ and $W'=V_1'+\cdots+ V_n' $ be
independent sums
of independent vectors $V_k=\sum_{j=1}^{s}\varepsilon_{jk}z_{jk}$, and
$V'_k=\sum_{j=1}^{s}\varepsilon'_{jk}z'_{jk}$, respectively, with $
\{(z_{1k},\ldots,z_{sk})$, $(z'_{1k},\ldots,z'_{sk}) \in
{\bolds{\Gamma}} (\delta;\mathbb D, \mathcal S ), k=1,\ldots
,n \}$.

Using that all random vectors $\widetilde V_k$ are symmetrized and have
nonnegative characteristic functions and applying H\"older's
inequality, we obtain, for each~$t$,
%
%
\begin{eqnarray}
\label{l1} \E\operatorname{e} \bigl\{ t \bigl\langle\mathbb Q
\widetilde W,
\widetilde W' \bigr\rangle\bigr\} &=& {\mathbf E}_{\widetilde W'}
\Biggl(\prod_{k=1}^n {\mathbf
E}_{\widetilde V_k}\operatorname{e} \bigl\{ t \bigl\langle\mathbb Q
\widetilde
V_k,\widetilde W' \bigr\rangle\bigr\} \Biggr)
\\
& \le& \Biggl(\prod_{k=1}^n {\mathbf
E}_{\widetilde W'} \bigl({\mathbf E}_{\widetilde V_k} \operatorname
{e} \bigl\{ t
\bigl\langle\mathbb Q \widetilde V_k,\widetilde W'
\bigr\rangle\bigr\} \bigr)^n \Biggr)^{1/n}
\\
&=& \Biggl(\prod_{k=1}^n {\mathbf
E}_{\widetilde W'} \bigl({\mathbf E}_{\widetilde T_k}\operatorname{e}
\bigl\{ t
\bigl\langle\mathbb Q \widetilde T_k,\widetilde W'
\bigr\rangle\bigr\} \bigr) \Biggr)^{1/n}
\\
& = & \Biggl(\prod_{k=1}^n \E
\operatorname{e} \bigl\{ t \bigl\langle\mathbb Q \widetilde T_k,
\widetilde W' \bigr\rangle\bigr\} \Biggr)^{1/n},\label{l2}
\end{eqnarray}
where $\widetilde T_k \=\sum_{l=1}^n \widetilde V_{kl}$ denotes a sum
of i.i.d. copies $\widetilde
V_{kl}$ of
$\widetilde V_k$ which are independent
of all other random vectors and variables.

Repeating the steps \eqref{l1}--\eqref{l2} for each factor $\E
\operatorname{e}
\{ t
\langle\mathbb Q \widetilde T_k,\widetilde W' \rangle\}$ instead of
the expectation $ \E
\operatorname{e} \{ t
\langle\mathbb Q \widetilde W,\widetilde W' \rangle\} $ on the
right-hand side
separately, we get (with $\widetilde T'_{i} \=\sum_{l=1}^{n}
\widetilde{V}'_{i
l}$, where $\widetilde{V}'_{i l}$ are i.i.d. copies of $\widetilde{V}'_{i}$
independent of all other random vectors)
%
%
\begin{equation}
\E\operatorname{e} \bigl\{ t \bigl\langle\mathbb Q \widetilde W
,\widetilde
W' \bigr\rangle\bigr\} \le\Biggl(\prod
_{k=1}^n \prod_{i=1}^{n}
\E\operatorname{e} \bigl\{ t \bigl\langle\mathbb Q \widetilde T_k,
\widetilde T'_i \bigr\rangle\bigr\}
\Biggr)^{1/n^2}.\label{qwe1}
\end{equation}

Thus, using \eqref{qwe1} and the arithmetic-geometric mean inequality,
we have
%
%
\begin{eqnarray}\label{qwe}
&&\int_A^B \sqrt{\E\operatorname{e} \bigl\{ t
\bigl\langle\mathbb Q \widetilde W,\widetilde W' \bigr\rangle/2
\bigr\}} \frac{ dt }{ | t | } \nonumber\\
&&\qquad\le\int_A^B
\Biggl(\prod_{k=1}^n \prod
_{i=1}^{n} \E\operatorname{e} \bigl\{ t \bigl\langle
\mathbb Q \widetilde T_k,\widetilde T'_i
\bigr\rangle/2\bigr\} \Biggr)^{1/2n^2} \frac{ dt }{ | t | }
\nonumber
\\[-8pt]
\\[-8pt]
\nonumber
&&\qquad\le \frac{ 1 }{ n^2 } \sum_{k=1}^n
\sum_{i=1}^{n} \int_A^B
\bigl(\E\operatorname{e} \bigl\{ t \bigl\langle\mathbb Q
\widetilde
T_k,\widetilde T'_i \bigr\rangle/2
\bigr\} \bigr)^{1/2} \frac{ dt }{ |
t | }
\\
&&\qquad\le \sup_\Gamma\int_A^B
\sqrt{\E\operatorname{e} \bigl\{ t \bigl\langle\mathbb Q
\widetilde T,\widetilde
T' \bigr\rangle/2 \bigr\}} \frac{ dt }{ | t | },\nonumber
\end{eqnarray}
where
$T= U_1+\cdots+ U_n$ and $T'=U_1'+\cdots+ U_n' $ are
independent sums of independent copies of random vectors
$U=\sum_{j=1}^{s}\varepsilon_{j1}z_{1}$ and
$U'=\sum_{j=1}^{s}\varepsilon'_{j1}z'_{1}$, and
$\sup_\Gamma$ is taken over all
$ (z_{1},\ldots,z_{s}), (z'_{1},\ldots,z'_{s}) \in
{\bolds{\Gamma}} (\delta;\mathbb D, \mathcal S )$.
Inequalities \eqref{eq71q}
and~\eqref{qwe} imply now the statement of the lemma.
\end{pf*}

The following Lemma \ref{Le32} provides a Poisson summation
formula.

%
\begin{lemma}\label{Le32}
Let $\operatorname{Re} z > 0, a,b \in{\mathbb R}^s$ and
$\mathbb S\dvtx{\mathbb R}^s \rightarrow\mathbb R^s$ be a positive
definite symmetric nondegenerate linear operator. Then
\begin{eqnarray*}
&&\sum_{m \in\mathbb{Z}^s} \exp\bigl\{-z \mathbb S[m+a] +
2 \pi i \langle m,b\rangle\bigr\}
\\
&&\qquad =  \bigl(\det(\mathbb S / \pi) \bigr)^{-1/2} z^{-s/2} \exp
\bigl\{ - 2 \pi i \langle a,b\rangle\bigr\} \\
&&{}\qquad\quad\times\sum
_{l
\in\mathbb{Z}^s} \exp\biggl\{- \frac{ \pi^2 }{ z } \mathbb
S^{-1}[l + b] -2\pi i \langle a, l\rangle\biggr\},
\end{eqnarray*}
where $\mathbb S^{-1}\dvtx{\mathbb R}^s \rightarrow\mathbb R^s$
denotes the inverse positive definite operator for $\mathbb S$.
\end{lemma}

\begin{pf}
See, for example, \citet{Fri82}, {page} 116, or \citet{Mum83},
{page}~189, formula (5.1); and {page} 197, formula (5.9).\vadjust{\goodbreak}
\end{pf}

Let the conditions of Lemma \ref{L73} be satisfied. Introduce
one-dimensional lattice probability distributions $H_n=\mathcal
L(\xi_n)$ with integer valued $\xi_n$ setting
\[
\P\{\xi_n=k \}= A_n n^{-1/2} \exp\bigl
\{-k^2/2n \bigr\} \qquad\mbox{for } k\in\mathbb Z.
\]
It is easy to see that
${A_n\asymp1}$. Moreover, by Lemma \ref{Le32},
%
%
\begin{equation}
\widehat H_n(t)\ge0 \qquad\mbox{for all } t\in\mathbb R.\label{eq76}
\end{equation}
Introduce the $s$-dimensional random vector $\zeta_n$ having as
coordinates independent copies of $\xi_n$. Then, for
$m=(m_1,\ldots,m_s)\in\mathbb Z^s$, we have
%
%
\begin{equation}
\label{qm} q(m)\=\P\{\zeta_n=m \}=A_n^s
n^{-s/2} \exp\bigl\{-\Vert m\Vert^2/2n \bigr\}.
\end{equation}

%
\begin{lemma}\label{L75}
Let
$W=V_1+\cdots+ V_n $ and $W'=V_1'+\cdots+ V_n' $
denote independent sums of independent copies of random vectors
$V$ and $V'$ such that
\[
V=\varepsilon_{11} z_1+ \cdots+\varepsilon_{s1}
z_s,\qquad V'=\varepsilon'_{11}
z_1'+ \cdots+\varepsilon'_{s1}
z_s',
\]
with some
$z_j,z_j'\in\mathbb R^d $. Introduce the matrix $\mathbb B_t= \{
b_{ij}(t)\dvtx1\leq i,j\leq s\}$ with $b_{ij}(t)= t \langle\mathbb
Qz_i,z_j' \rangle$. Then
\[
\E\operatorname{e} \bigl\{ t \bigl\langle\mathbb Q\widetilde W
,\widetilde
W' \bigr\rangle/4 \bigr\} \ll_s\E\operatorname{e}
\bigl\{ \bigl\langle\mathbb B_t \zeta_n,
\zeta'_n \bigr\rangle\bigr\}+ \exp\{-c n \}\qquad
\mbox{ for all } t\in\mathbb R,
\]
where $\zeta'_n$ are independent copies of $\zeta_n$ and $c$ is a
positive absolute constant.
\end{lemma}

\begin{pf}
Without loss of generality, we assume that $n\ge
c_1$, with a sufficiently large absolute constant $c_1$. Consider
the random vector $Y=(\widetilde\varepsilon_1,\ldots,\break \widetilde
\varepsilon_s)\in\mathbb R^s$
with coordinates which are symmetrizations of i.i.d. Rademacher
random variables. Let $R=(R_1, \ldots,R_s)$ and $T$ denote
independent sums of $n$ independent copies of $Y/2$. Then we can
write
%
%
\begin{equation}
\label{rav} \E\operatorname{e} \bigl\{ t \bigl\langle\mathbb
Q\widetilde W,
\widetilde W' \bigr\rangle/4 \bigr\} = \E\operatorname{e} \bigl\{
\langle\mathbb B_t R,T \rangle\bigr\}\qquad \mbox{ for all } t\in
\mathbb R.
\end{equation}
Note that the scalar product
$\langle\cdt, \cdt\rangle$ in $\E\operatorname{e} \{
\langle\mathbb B_t R,T \rangle\}$ means the scalar product
of vectors in $\mathbb R^s$. In order to estimate this
expectation, we write it in the form
%
%
\begin{eqnarray}
\E\operatorname{e} \bigl\{ \langle\mathbb B_t R,T \rangle\bigr
\} &=& \E{\mathbf E}_R \operatorname{e} \bigl\{ \langle\mathbb
B_t R,T \rangle\bigr\}
\nonumber
\\
&=&\sum_{{\ov m}\in\mathbb Z^s}p({\ov m})\sum
_{m\in\mathbb Z^s}p(m) \operatorname{e} \bigl\{ \langle\mathbb
B_t m,{\ov m} \rangle\bigr\}, \label{eq77}
\end{eqnarray}
with summing over $m=(m_1,\ldots,m_s)\in\mathbb Z^s$,
$\ov{m}=(\ov{m}_1,\ldots,\ov{m}_s)\in\mathbb Z^s$ and
%
%
\begin{equation}
\label{pm} p(m)=\P\{R=m \}=\prod_{j=1}^s
\P\{R_j=m_j \} =\prod_{j=1}^s
2^{-2n}\pmatrix{2n\cr
m_j+n},
\end{equation}
if
$\max_{1\le j\le s}|m_j|\le n$ and $p(m)=0$ otherwise.
Clearly, for fixed $T=\ov{m}$,
%
%
\begin{equation}
{\mathbf E}_R \operatorname{e} \bigl\{ \langle\mathbb
B_t R,T \rangle\bigr\} =\sum_{m\in\mathbb Z^s}p(m)
\operatorname{e} \bigl\{ \langle\mathbb B_t m,{\ov m} \rangle
\bigr\}\ge0 \label{eq78}
\end{equation}
is a value of the characteristic function of symmetrized
random vector $\mathbb B_t R$. Using Stirling's formula, it is
easy to show that there exist positive absolute constants $c_2$
and $c_3$ such that
%
%
\begin{equation}
\P\{R_j=m_j \}\ll n^{-1/2} \exp\bigl
\{-m_j^2/2n \bigr\} \qquad\mbox{for } |m_j|\le
c_2n\label{eq79}
\end{equation}
and
%
%
\begin{equation}
\P\bigl\{|R_j|\ge c_2n \bigr\}\ll\exp\{-c_3n \}.
\label{eq710}
\end{equation}
Using \eqref{eq77}--\eqref{eq710}, we obtain
%
%
\begin{eqnarray}\label{eq712}
\E\operatorname{e} \bigl\{ \langle\mathbb B_t R,T \rangle
\bigr\} &\ll_s&\sum_{{\ov m}\in\mathbb Z^s}q(\ov m) \sum
_{m\in\mathbb Z^s}p(m) \operatorname{e} \bigl\{ \langle\mathbb
B_t m,{\ov m} \rangle\bigr\}+ \exp\{-c_3n \}
\nonumber
\\
& =&\sum_{m\in\mathbb Z^s}p(m) \sum
_{{\ov m}\in\mathbb Z^s}q(\ov m) \operatorname{e} \bigl\{ \langle
\mathbb
B_t m,{\ov m} \rangle\bigr\}+ \exp\{-c_3n \}
\nonumber
\\[-8pt]
\\[-8pt]
\nonumber
& =&\E{\mathbf E}_{\zeta_n} \operatorname{e} \bigl\{ \langle
\mathbb
B_t R,\zeta_n \rangle\bigr\}+ \exp
\{-c_3n \}
\\
& =&\E\operatorname{e} \bigl\{ \langle\mathbb B_t R,
\zeta_n \rangle\bigr\}+ \exp\{-c_3n \}.
\nonumber
\end{eqnarray}
Now we repeat our previous arguments, noting that
%
%
\begin{equation}
{\mathbf E}_{\zeta_n} \operatorname{e} \bigl\{ \langle\mathbb
B_t R,\zeta_n \rangle\bigr\} =\sum
_{\ov m\in\mathbb Z^s}q(\ov m) \operatorname{e} \bigl\{ \langle
\mathbb
B_t R,{\ov m} \rangle\bigr\}\ge0 \label{eq713}
\end{equation}
is a value of the
nonnegative characteristic function of the random vector
$\zeta_n$; see~\eqref{eq76}. Using again \eqref{eq79} and
\eqref{eq710}, we obtain
%
%
\begin{equation}
\E\operatorname{e} \bigl\{ \langle\mathbb B_t R,
\zeta_n \rangle\bigr\} \ll_s\E\operatorname{e} \bigl\{
\bigl\langle\mathbb B_t \zeta_n,
\zeta'_n \bigr\rangle\bigr\}+ \exp\{-c_3
n \}. \label{eq714}
\end{equation}
Relations \eqref{rav}, \eqref{eq712} and
\eqref{eq714} imply the statement of the lemma.
\end{pf}

Let us estimate the expectation $ \E\operatorname{e} \{
\langle\mathbb B_t \zeta_n,\zeta'_n \rangle\}$ under the
conditions of Lemmas~\ref{L73} and~\ref{L75}, assuming that
$s=d$, $\mathbb D=\mathbb C^{-1/2}$, $\delta\leq1/(5 s)$, $n\ge
c_4$, where $c_4$ is a sufficiently large absolute constant, and $
(z_{1},\ldots,z_{s}), (z'_{1},\ldots,z'_{s}) \in
{\bolds{\Gamma}} (\delta;\mathbb D, \mathcal S )$, that is,
%
%
\begin{equation}
\bigl\|\mathbb C^{-1/2}z_j-e_j\bigr\|\leq\delta,\qquad \bigl\|
\mathbb C^{-1/2}z_j'-e_j\bigr\|\leq
\delta \qquad\mbox{for } 1\leq j\leq s,\label{eq7.6}
\end{equation}
with an orthonormal system $\mathcal
S=\mathcal S_o= \{e_1,\ldots,e_s \} $ involved in the
conditions of Lemma \ref{L73}. We can rewrite $\E\operatorname
{e} \{ \langle\mathbb B_t \zeta_n,\zeta'_n \rangle\}$
as
\[
\E\operatorname{e} \bigl\{ \bigl\langle\mathbb B_t
\zeta_n,\zeta'_n \bigr\rangle\bigr\} =
\sum_{{\ov m}\in\mathbb Z^s}q({\ov m})\sum
_{m\in
\mathbb
Z^s}q(m) \operatorname{e} \bigl\{ \langle\mathbb
B_t \ov m,m \rangle\bigr\}.
\]
Thus, by \eqref{qm},
\[
\E\operatorname{e} \bigl\{ \bigl\langle\mathbb B_t
\zeta_n,\zeta'_n \bigr\rangle\bigr\}
=A_n^{2s} n^{-s} \sum
_{{\ov m}\in\mathbb
Z^s}\sum_{m\in\mathbb Z^s} \exp\bigl\{ i
\langle\mathbb B_t \ov m,{m}\rangle-\Vert m\Vert^2/2n
-\Vert\ov m\Vert^2/2n \bigr\}.
\]

Denote
%
%
\begin{equation}
\label{defr}r=\sqrt{2 \pi^2n}.
\end{equation}
Applying Lemma \ref{Le32} with $\mathbb S=\mathbb I_s$,
$z=1/2n$, $a=0$, $b=(2\pi)^{-1} \mathbb B_t \ov m $ and
using that ${A_n\asymp1}$, we obtain
%
%
\begin{eqnarray}\label{koren1}
\label{koren} &&\E\operatorname{e} \bigl\{ \bigl\langle\mathbb B_t
\zeta_n,\zeta'_n \bigr\rangle\bigr\}\nonumber \\
&&\qquad
\ll_s n^{-s/2} \sum_{l,m \in\mathbb{Z}^s} \exp
\bigl\{-2 \pi^2n \bigl\Vert l+(2\pi)^{-1}\mathbb
B_t m\bigr\Vert^2-\Vert m\Vert^2/2n \bigr\}
\\
&&\qquad\ll_s r^{-s}\sum_{m,\ov m \in\mathbb{Z}^{s}} \exp
\bigl\{-r^2 \Vert m-t \mathbb V \ov m\Vert^2-\Vert\ov
m\Vert^2/r^2 \bigr\},\nonumber
\end{eqnarray}
where $\mathbb V\dvtx\mathbb
R^{s}\to\mathbb R^{s}$ is the operator with matrix
%
%
\begin{equation}
\label{bbbl} \mathbb V=(2\pi)^{-1}\mathbb B_1.
\end{equation}
Note that
the right-hand side of \eqref{koren1} may be considered as a
theta-series.

Denote $y_k=\mathbb C^{-1/2}z_k$, $1\leq k\leq s$. Let $\mathbb Y$
be the $(s\times s)$-matrix with entries $\langle e_j,
y_k\rangle$, where index $j$ is the number of the row, while $k$
is the number of the column. Then the matrix $\mathbb F\=\mathbb
Y^*\mathbb Y$ has entries $\langle y_j, y_k\rangle$. Here
$\mathbb Y^*$ is the transposed matrix for $\mathbb Y$. According
to \eqref{eq7.6}, we have
%
%
\begin{equation}
\|y_j-e_j\|\leq\delta\qquad \mbox{for } 1\leq j\leq s.
\label{eq76o}
\end{equation}

Let us show that [cf. BG (\citeyear{BenGot97N1}), proof of Lemma 7.4]
%
%
\begin{equation}
\label{bbl} \Vert\mathbb Y\Vert\le3/2 \quad\mbox{and}\quad \Vert\mathbb Y^{-1}
\Vert\le2.
\end{equation}
Since $\mathcal
S_o= \{e_1,e_2,\ldots,e_s \} $ is an orthonormal system,
inequalities \eqref{eq76o} imply that $\mathbb Y=\mathbb
I_s+\mathbb A$ with some matrix $\mathbb A=\{a_{ij}\}$ such that
$|a_{ij}|\leq\delta$. Thus, we have $\|\mathbb A\| \leq\|\mathbb
A\|_2\leq s \delta$, where $\|\mathbb A\|_2$ denotes the
Hilbert--Schmidt norm of the matrix~$\mathbb A$. Therefore, the
condition $\delta\leq1/(5 s)$ implies $\|\mathbb A\| \leq1/2$ and
inequalities \eqref{bbl}.

The matrix $\mathbb F$ is symmetric and positive definite. Its
determinant is the product of eigenvalues which [by \eqref{bbl}]
are bounded from above and from below by some absolute positive
constants. Moreover,
%
%
\begin{equation}
\label{dett3} (\det\mathbb Y )^2= \bigl(\det\mathbb Y^*
\bigr)^2=\det\mathbb F \asymp_s1\asymp\Vert\mathbb F\Vert
\asymp\Vert\mathbb Y\Vert.
\end{equation}
Define
the matrices $\ov{\mathbb Y}$ and $\ov{\mathbb F}$, replacing
$z_j$ by $z_j'$ in the definition of ${\mathbb Y}$ and ${\mathbb
F}$.
Similarly to \eqref{dett3}, one can show that
%
%
\begin{equation}
\label{dett4} (\det\ov{\mathbb Y} )^2= \bigl(\det\ov{\mathbb Y} ^*
\bigr)^2=\det\ov{\mathbb F} \asymp_s1\asymp\Vert\ov{
\mathbb F}\Vert\asymp\Vert\ov{\mathbb Y}\Vert.
\end{equation}
Let $\mathbb G$ and
$\ov{\mathbb G}$ be the $(s\times s)$-matrices with entries
$\langle e_j,\mathbb Qz _k\rangle$ and $\langle e_j,
z'_k\rangle$, respectively. Then, clearly, $\mathbb G=\mathbb
Q\mathbb C^{1/2}\mathbb Y$ and $\ov{\mathbb G}=\mathbb
C^{1/2}\ov{\mathbb Y}$. Therefore,
%
%
\begin{equation}
\label{dett5} \mathbb B_1=\mathbb G^*\ov{\mathbb G}=\mathbb Y^*
\mathbb C^{1/2}\mathbb Q\mathbb C^{1/2}\ov{\mathbb Y}.
\end{equation}

Moreover, $\mathbb Q^2= \mathbb I_d$ implies that $ |\det
\mathbb Q |= 1$ and $\Vert\mathbb Q\Vert=1$.
Using relations \eqref{bbbl}
and \eqref{dett3}--\eqref{dett5}, we
obtain
%
%
\begin{equation}
\label{ett7} |\det\mathbb V |\asymp_s |\det\mathbb B_1
|\asymp_s\det\mathbb C
\end{equation}
and
%
%
\begin{equation}
\label{ett8} \Vert\mathbb V\Vert\ll\Vert\mathbb B_1\Vert\ll\Vert
\mathbb C\Vert\ll\sigma_1^2.
\end{equation}
%

\section{Some facts from number theory}
\label{s6}

In Section \ref{s6}, we consider some facts of the geometry of
numbers; see \citet{Dav58} or~\citet{Cas59}. They will help
us to estimate the integrals of the right-hand side of inequality~\eqref{koren1}. See G\"otze and Margulis (\citeyear{GotMar10}) or G\"
otze and
Zaitsev (\citeyear{GotZai10}) for a more detailed version of this section.

Let $e_1, e_2,\ldots, e_d$ be linearly independent vectors in
$\R^d$. The set
%
%
\begin{equation}
\Lambda= \Biggl\{{ 
\sum_{j=1}^d}
n_j e_j\dvtx n_j\in\mathbb Z, j=1,2,\ldots,d
\Biggr\}
\end{equation}
is called the lattice with basis $e_1, e_2,\ldots, e_d$. The
determinant $\det(\Lambda)$ of a lattice $\Lambda$ is the modulus
of the determinant of the matrix formed from the vectors $e_1,
e_2,\ldots, e_d$. If $\Lambda=\mathbb A \mathbb Z^d$, where
$\mathbb A$ is a nondegenerate linear operator, then $\det
(\Lambda)=\llvert\det\mathbb A\rrvert$.

Let $F\dvtx{\R}^d \rightarrow[0, \infty)$ denote a norm on
${\R}^d$. The successive minima $M_1 \leq\cdots\leq M_d$ of $F$
with respect to a lattice $\Lambda\subset{\R}^d$ are defined as
follows: $M_j$ is the infimum of $\lambda> 0$ such that the set
$ \{m \in\Lambda\dvtx F(m) < \lambda\}$ contains $j$
linearly independent vectors. The following Lemma \ref{Dav2} is
proved by Davenport [(\citeyear{Dav58}), Lemma 1] for $\Lambda={\Z
}^d$; see also
G\"otze and Margulis (\citeyear{GotMar10}).

%
\begin{lemma} \label{Dav2}
Let $M_1 \leq\cdots\leq M_d$
be the successive minima of a norm $F$ with
respect to a lattice $\Lambda\subset{\R}^d$. Denote
$M_{d+1}=\infty$. Suppose that $1\le j\le d$ and $M_j \leq b \leq
M_{j+1}$, for some ${b>0}$.
Then
%
%
\begin{equation}
\# \bigl\{m=(m_1, \ldots, m_d) \in\mathbb{Z}^d\dvtx F(m) < b \bigr
\} \asymp_d b^j (M_1\cdt
M_2 \cdots M_j)^{-1}.
\end{equation}
\end{lemma}

Representing $\Lambda=\mathbb A \mathbb Z^d$, we see that the
lattice $\Lambda=\mathbb Z^d$ may be replaced in Lemma \ref{Dav2}
by any lattice $\Lambda\subset\R^d$.

%
\begin{lemma} \label{Dav9}
Let $ F_j (m)$, $j=1,2$, be some norms in ${\R}^d$ and $M_1
\leq\cdots\leq M_d$ and $N_1 \leq\cdots\leq N_d$ be the
successive minima of $F_1$ with respect to a lattice $\Lambda_1$
and of $F_2$ with respect to a lattice $\Lambda_2$, respectively.
Let $C>0$. Assume that $M_k\gg_d C F_2(n_k)$, $k=1,2,\ldots,d$, for some
linearly independent vectors $n_1,n_2,\ldots,n_d\in\Lambda_2$.
Then
%
%
\begin{equation}
M_k\gg_d C N_k, \qquad k=1,\ldots,d.
\end{equation}
\end{lemma}

%
\begin{lemma} \label{Dav5}
Let $\Lambda$ be a lattice in ${\R}^d$ and let $c_j(d)$,
$j=1,2, 3$, be positive quantities depending on $d$ only. Let
$ F (\cdt)$ be a norm in ${\R}^d$ such that $F(\cdt)\asymp_d
\Vert\cdt\Vert$. Then
%
%
\begin{eqnarray}
\sum_{v \in\Lambda} \exp\bigl\{- c_1(d) \Vert v
\Vert^2 \bigr\} &\asymp_d& \sum
_{v \in\Lambda} \exp\bigl\{- c_2(d) \bigl(F(v)
\bigr)^2 \bigr\}
\nonumber
\\[-8pt]
\\[-8pt]
\nonumber
&\asymp_d&\# \bigl\{v \in\Lambda\dvtx F(v)< c_3(d) \bigr\}.
\end{eqnarray}
\end{lemma}

For a lattice $\Lambda\subset\mathbb R^{d}$ and $1\le l\le d$, we
define its $\alpha_l$-characteristics by
%
%
\begin{equation}
\label{alp} \alpha_l( \Lambda)\= \sup\bigl\{ \bigl|\det\bigl(
\Lambda'\bigr)\bigr |^{-1}\dvtx\Lambda' \mbox{ is
a $l$-dimensional sublattice of $\Lambda$} \bigr\}.
\end{equation}
Denote
%
%
\begin{equation}
\label{alp3} \alpha( \Lambda)\= \max_{1\le l\le d}
\alpha_l( \Lambda).
\end{equation}

%
\begin{lemma} \label{Dav4}
Let $ F (\cdt)$ be a norm in ${\R}^d$ such that $F(\cdt)\asymp_d
\Vert\cdt\Vert$. Let $c(d)$
be a positive quantity depending on $d$ only.
Let $M_1 \leq\cdots\leq M_d$ be the successive minima of $F$ with
respect to a lattice $\Lambda\subset\R^d$. Then
%
%
\begin{equation}
\label{LLL1}\alpha_l( \Lambda)\asymp_d (M_1
\cdt M_2 \cdots M_l)^{-1},\qquad  l=1,\ldots,d.
\end{equation}
Moreover,
%
%
\begin{equation}
\label{LLL2}\alpha( \Lambda)\asymp_d \# \bigl\{v \in\Lambda
\dvtx \Vert v
\Vert< c(d) \bigr\},
\end{equation}
provided that
$M_1\ll_d1$.
\end{lemma}

Lemma \ref{Dav4} is an easy consequence of the following lemma
formulated in proposition (page 517) and remark (page 518) in Lenstra, Lenstra and Lov\'asz (\citeyear{LenLenLov82}).

%
\begin{lemma} \label{LLL}
Let $M_1 \leq\cdots\leq M_d$ be the successive minima of the
standard Euclidean norm with respect to a lattice
$\Lambda\subset\R^d$. Then there exists a basis $e_1, e_2,\ldots,
e_d$ of $\Lambda$ such that
%
%
\begin{equation}
M_l\asymp_d\Vert e_l\Vert,\qquad  l=1,\ldots,d.
\end{equation}
Moreover,
%
%
\begin{equation}
\det(\Lambda)\asymp_d\prod_{l=1}^d
\Vert e_l\Vert.
\end{equation}
\end{lemma}

\section{From number theory to probability}
\label{s7}

In Section \ref{s7}, we use number-theoretical results of Section \ref
{s6} to estimate integrals of the right-hand side of~\eqref{koren1}. Recall that we have assumed the conditions of
Lemmas \ref{L73} and \ref{L75}, $s=d$, $\mathbb D=\mathbb
C^{-1/2}$, $ \delta\leq1/(5 s)$, $n\ge c_4$ and \eqref{eq76},
for an orthonormal system $\mathcal S=\mathcal S_o$. The notation
$\operatorname{SL}(d, \R)$ is used below for the set of all $(d\times
d)$-matrices with real entries and determinant 1.

Introduce the matrices
%
%
\begin{eqnarray}
\label{svo4} \mathbb D_r &\= &\pmatrix{ r \mathbb I_{s} & \mathbb O_{s}
\cr
\mathbb O_{s} & r^{-1} \mathbb I_{s}
}\in \operatorname{SL}(2s,\R),\qquad  r>0,
\\
\label{svon} \mathbb K_t& \= &\pmatrix{ \mathbb I_{s} & -t \mathbb I_{s}
\cr
t \mathbb I_{s} & \mathbb I_{s}},\qquad  t\in\R,
\\
\label{svo5} \mathbb U_t &\=& \pmatrix{ \mathbb I_{s} & -t \mathbb I_{s}
\cr
\mathbb O_{s} & \mathbb I_{s}}\in \operatorname{SL}(2s,\R),\qquad t\in\R,
\end{eqnarray}
and the lattices
%
%
\begin{eqnarray}
\label{latt} \Lambda&\=&\pmatrix{ \mathbb
I_{s}&\mathbb O_{s}
\cr
\mathbb O_{s} &\mathbb V_0}\mathbb Z^{2s},
\\
\label{jj}\Lambda_{j}&=& \mathbb D_{j} \mathbb
U_{j^{-1}} \Lambda=\pmatrix{ j
\mathbb I_{s}&\mathbb-\mathbb V_0
\cr
\mathbb O_{s} &j^{-1} \mathbb V_0
}\mathbb Z^{2s},\qquad j=1,2,\ldots,
\end{eqnarray}
where
%
%
\begin{equation}
\label{svo9} \mathbb V_0=\sigma_1^{-2}
\mathbb V
\end{equation}
and
the matrix $\mathbb V$ is defined in \eqref{bbbl}.
Below we use the following simplest properties of
these matrices:
%
%
\begin{eqnarray}
\label{svo} \mathbb D_a\mathbb D_b=\mathbb
D_{ab},\qquad \mathbb U_a\mathbb U_b=\mathbb
U_{a+b} \quad\mbox{and} \quad\mathbb D_a \mathbb
U_b = \mathbb U_{a^2b} \mathbb D_a
\nonumber
\\[-8pt]
\\[-8pt]
\eqntext{\mbox{for $a,b>0$.}}
\end{eqnarray}

Let $\Vert x\Vert_\infty=\max_{1\le j\le d}|x_j|$, for $x=({x}_1,
\ldots, {x}_d)\in\mathbb R^d$. Let $M_{j,t}$, $j=1,2, \ldots,\break  2s$,
be the
successive minima of the norm $\Vert\cdt\Vert_\infty$ with respect to
the lattice
%
%
\begin{equation}
\label{latt8} \Xi_t\=\pmatrix{r \mathbb I_{s}&-rt \mathbb V
\cr
\mathbb O_{s} &r^{-1} \mathbb I_{s}
}\mathbb Z^{2s}.
\end{equation}
Moreover,
simultaneously, $M_{j,t}$ are the successive minima of the norm
$F^*(\cdt)$ defined for $(m,\ov m)\in\mathbb R^{2s}$, $m,\ov
m\in\mathbb R^{s}$, by
%
%
\begin{equation}
\label{latt5} F^*\bigl((m,\ov m)\bigr)\=\max\bigl\{ \Vert m
\Vert_\infty, \sigma_1^{2} \bigl\Vert\mathbb
V^{-1}\ov m\bigr\Vert_\infty\bigr\}
\end{equation}
with respect to the lattice
%
%
\begin{equation}
\label{latt6} \Omega_t\=\pmatrix{ r \mathbb I_{s}&-rt \mathbb V
\cr
\mathbb O_{s} &\sigma_1^{-2} r^{-1}
\mathbb V
}\mathbb Z^{2s} =\mathbb
D_r\mathbb U_u \Lambda\qquad\mbox{where }u\=
\sigma_1^{2} t.
\end{equation}
Using Lemmas
\ref{Dav9} and \ref{LLL} and the equality $\det(\Xi_t)=1$, it is easy
to show that
%
%
\begin{equation}
\label{tyy1} M_{1,t}\ll_s 1.
\end{equation}

Let $M_{j,t}^*$ be the successive minima of the Euclidean
norm
with respect to the lattice $\Omega_t$. Note that, according to
\eqref{ett8} and \eqref{latt5},
%
%
\begin{equation}
\label{att}\Vert\cdt\Vert\ll_{s} F^*(\cdt).
\end{equation}
Using \eqref{att} and Lemma \ref{Dav9},
we obtain
%
%
\begin{equation}
\label{att88} M_{j, t}^*\ll_s M_{j, t},\qquad  j=1,
\ldots,2s.
\end{equation}
According to Lemma \ref{Dav4},
%
%
\begin{equation}
\label{att89}\alpha(\Xi_t) \ll_s\alpha(
\Omega_t).
\end{equation}

Let us estimate $\alpha(\Omega_t)$ assuming that $r\ge1$ and (for
a moment) $t=\sigma_1^{-2} r^{-1}$. In this case
%
%
\begin{equation}
\label{latt36} \Omega_t=\pmatrix{ r \mathbb I_{s}&-\mathbb V_0
\cr
\mathbb O_{s} &r^{-1} \mathbb V_0}\mathbb Z^{2s}.
\end{equation}
By relation \eqref{LLL2} of Lemma \ref{Dav4}, we have
%
%
\begin{equation}
\alpha(\Omega_t) \asymp_s\#\bigl \{v \in
\Omega_t \dvtx\Vert v\Vert< 1/2 \bigr\}=\#K, \label{att889}
\end{equation}
where
%
%
\begin{eqnarray}\label{a889}
&&K= \bigl\{v=(m, \ov m)\in\mathbb Z^{2s}\dvtx m, \ov m\in\mathbb
Z^{s},
\nonumber
\\[-8pt]
\\[-8pt]
\nonumber
&&\hspace*{28pt}\Vert rm-\mathbb V_0\ov m\Vert^2+
\bigl\Vert r^{-1}\mathbb V_0\ov m\bigr\Vert^2< 1/4
\bigr\}.
\end{eqnarray}
Let us estimate from above the
right-hand side of \eqref{att889}. If $v=(m, \ov m)\in K$, then
%
%
\begin{equation}
r \Vert m\Vert\le\Vert rm-\mathbb V_0\ov m\Vert+\Vert\mathbb
V_0\ov m\Vert< \frac{ 1 }{ 2 } + \frac{ r }{ 2 } \le r.
\label{ad889}
\end{equation}
Hence
$m=0$ and $\Vert\mathbb V_0\ov m\Vert\le1/2$. It remains to
estimate the quantity
%
%
\begin{equation}
R\=\# \bigl\{ \ov m \in\mathbb Z^{s}\dvtx\Vert\mathbb V_0
\ov m\Vert< 1 \bigr\}\ge\# K. \label{ada889}
\end{equation}
Let $N_1\le\cdots\le N_{s}$ be the successive minima of the
Euclidean norm
with respect to the lattice $\mathbb V_0\mathbb Z^{s}$.
Let $e_1,e_2,\ldots, e_{s}$ be the standard orthonormal basis of~$\mathbb Z^{s}$.
By~\eqref{ett8} and \eqref{svo9}, we have
$\Vert\mathbb V_0e_j\Vert\le1$, $j=1,2, \ldots, s$. Therefore,
using Lemma~\ref{Dav9}, we see that $N_1\le\cdots\le N_{s}\le1$.
By \eqref{ett7}, \eqref{svo9}, \eqref{ada889} and by Lemmas
\ref{Dav2}, \ref{Dav9} and \ref{LLL},
%
%
\begin{equation}
\label{f99} R\asymp_s (N_1\cdt N_2 \cdots
N_s)^{-1}\asymp_s(\det\mathbb
V_0)^{-1}\asymp_s\sigma_1^{2s}
(\det\mathbb C)^{-1}.
\end{equation}
Hence, using \eqref{att889}, \eqref{ada889}
and \eqref{f99}, we conclude that
%
%
\begin{equation}
\label{adas89} \alpha(\Omega_t) \ll_s
\sigma_1^{2s} (\det\mathbb C)^{-1}\qquad \mbox{for }
r\ge1 \mbox{ and } t=\sigma_1^{-2} r^{-1}.
\end{equation}

Let now $t\in\mathbb R$ be arbitrary. By \eqref{latt8},
\eqref{tyy1}, \eqref{att89} and by
Lemmas \ref{Dav2}, \ref{Dav5} and~\ref{Dav4},
%
%
\begin{eqnarray}\label{kore}
\label{ken4}&&\sum_{m,\ov m \in\mathbb{Z}^{s}} \exp\bigl\{-r^2
\bigl\Vert m-t \mathbb V \ov m\bigr\Vert^2 - \Vert\ov m
\Vert^2/r^2 \bigr\}=\sum_{v \in\Xi_t}
\exp\bigl\{- \Vert v\Vert^2 \bigr\}
\nonumber
\\
&&\qquad\ll_sR_t\=\# \bigl\{v \in\Xi_t \dvtx\Vert v
\Vert< 1 \bigr\}
\\
&&\qquad \ll_s \alpha(\Xi_t) \ll_s\alpha(
\Omega_t).\nonumber
\end{eqnarray}
Now, by \eqref{koren}, \eqref{latt6} and
\eqref{ken4}, we have
%
%
\begin{eqnarray}
\label{koren4}\qquad \E\operatorname{e} \bigl\{ \bigl\langle\mathbb B_t
\zeta_n,\zeta'_n \bigr\rangle\bigr\}
\ll_s r^{-s} \alpha(\Omega_t)
=r^{-s} \alpha(\mathbb D_r\mathbb U_u
\Lambda) \qquad\mbox{where }u=\sigma_1^{2} t.
\end{eqnarray}

Let us estimate the quantity $R_t$, $t\in\mathbb R$, defined in
\eqref{kore} assuming that $r\ge1$ and $\llvert rt\rrvert\le
c_s^* \sigma_1^{-2}$, where $c_s^*\ge1$ is an arbitrary quantity
depending on $s$ only. By Lemma \ref{Dav5}, we have
%
%
\begin{equation}
\label{ad888}R_t\asymp_s\# K_0,
\end{equation}
where
%
%
\begin{eqnarray}\label{ba889}
&&K_0\= \bigl\{v=(m, \ov m)\in\mathbb Z^{2s}\dvtx m, \ov m\in
\mathbb Z^{s},
\nonumber
\\[-8pt]
\\[-8pt]
\nonumber
&&\hspace*{33pt}\Vert rm-rt \mathbb V\ov m\Vert^2+
\bigl\Vert r^{-1}\ov m\bigr\Vert^2< \bigl(2c_s^*
\bigr)^{-2} \bigr\}.
\end{eqnarray}
If $v=(m, \ov m)\in
K_0$, $r\ge1$ and $\llvert rt\rrvert\le c_s^* \sigma_1^{-2}$,
then, by \eqref{ett8} and \eqref{ba889},
%
%
\begin{equation}
r \Vert m\Vert\le\Vert rm-rt \mathbb V\ov m\Vert+\llvert rt\rrvert
\Vert
\mathbb V\ov m\Vert< \frac{ 1
}{ 2 } + \frac{ r }{ 2 } \le r.
\label{polk}
\end{equation}
Hence
$m=0$ and $\llvert rt\rrvert\Vert\mathbb V\ov m\Vert\le(2
c_s^*)^{-1}<1$. It remains to estimate the quantity
%
%
\begin{equation}
S\=\# \bigl\{ \ov m \in\mathbb Z^{s}\dvtx\llvert rt\rrvert\Vert
\mathbb V\ov m\Vert< 1 \bigr\}\ge\# K_0. \label{fdfd}
\end{equation}
Let
$P_1\le\cdots\le P_{s}$ be the successive minima of the Euclidean
norm
with respect to the lattice $\llvert rt\rrvert\mathbb V\mathbb Z^{s}$.
Let $e_1,e_2,\ldots, e_{s}$ be the standard orthonormal basis of
$\mathbb Z^{s}$. By~\eqref{ett8}, we have
$\Vert\llvert rt\rrvert\mathbb Ve_j\Vert\ll_s 1$, $j=1,2,
\ldots,
s$. Therefore, using Lemma \ref{Dav9}, we see that
$P_1\le\cdots\le P_{s}\ll_s 1$. By \eqref{ett7}, \eqref{fdfd}
and Lemmas \ref{Dav2} and \ref{LLL},
%
%
\begin{equation}
\label{yuyu} S\asymp_s (P_1\cdt P_2 \cdots
P_{s})^{-1}\asymp_s\bigl(\det\bigl(\llvert rt
\rrvert\mathbb V\bigr)\bigr)^{-1}\asymp_s\llvert rt
\rrvert^{-s} (\det\mathbb C)^{-1}.
\end{equation}
Hence, using \eqref{ad888}, \eqref{fdfd} and \eqref{yuyu},
we conclude that
%
%
\begin{equation}
\label{ghgh} R_t \ll_s\llvert rt\rrvert
^{-s} (\det\mathbb C)^{-1}\qquad\mbox{for } r\ge1 \mbox{ and }
\llvert rt\rrvert\le c_s^* \sigma_1^{-2}.
\end{equation}

Now, by \eqref{koren},
\eqref{ken4} and
\eqref{ghgh}, we have
%
%
\begin{eqnarray}
\label{koren44} \E\operatorname{e} \bigl\{ \bigl\langle\mathbb B_t
\zeta_n,\zeta'_n \bigr\rangle\bigr\} &
\ll_s& r^{-s} R_t
\nonumber\hspace*{-35pt}
\\[-8pt]
\\[-8pt]
\nonumber
&\ll_s& r^{-2s} \llvert t\rrvert^{-s} (\det
\mathbb C)^{-1}\qquad \mbox{for } r\ge1 \mbox{ and }\llvert rt\rrvert\le
c_s^* \sigma_1^{-2}.\hspace*{-35pt}
\end{eqnarray}
It is easy to
verify that
%
%
\begin{equation}
\int_{c_s\sigma_1^{-2}r^{-2+4/s}
}^{\sigma_1^{-2} r^{-1}}\sqrt{r^{-2s} \llvert t
\rrvert^{-s} (\det\mathbb C)^{-1}} \frac{ dt }{ t }
\ll_s r^{-2} \sigma_1^{s} (\det
\mathbb C)^{-1/2} \label{eq71av}
\end{equation}
for any $c_s$ depending on
$s$ only. Note that $ \sigma_1^{s} (\det\mathbb C)^{-
1/2}\ge1$. Using \eqref{koren4}, \eqref{eq71av} and Lemmas
\ref{L73} and \ref{L75},
we derive the following lemma.

%
\begin{lemma}\label{GZ} Let the conditions of Lemma \textup{\ref{L73}}
be satisfied with $s=d$, $\mathbb D=\mathbb C^{-1/2}$, $\delta\leq
1/(5 s)$ and with an orthonormal system $\mathcal S=\mathcal
S_o=\{ {e}_{1},\ldots, {e}_{s} \}\subset\mathbb R^d$. Let $c_s$ be an
arbitrary quantity
depending on $s$ only. Then, for any $b\in\mathbb R^d$ and $r\ge1$,
%
%
\begin{eqnarray}\label{eq71a}
&&\int_{c_s\sigma_1^{-2}r^{-2+4/s} }^{\sigma
_1^{-2}} \bigl| \widehat\Psi_b
(t/2) \bigr| \frac{ dt }{ t }
\nonumber\hspace*{-35pt}
\\[-8pt]
\\[-8pt]
\nonumber
&&\qquad \ll_s (p N)^{-1}
\sigma_1^{s} (\det\mathbb C)^{-1/2}+r^{-
s/2}
\sup_\Gamma\int_{r^{-1} }^1\bigl(
\alpha(\mathbb D_r\mathbb U_u \Lambda)
\bigr)^{1/2} \frac{ du }{ u },\hspace*{-35pt}
\end{eqnarray}
where $r$, $\alpha(\cdt)$, $\mathbb D_r$
$\mathbb U_t$ and the lattice $\Lambda$ are defined in relations
\eqref{dfn}, \eqref{defr}, \eqref{bbbl}, \eqref{alp}, \eqref
{alp3}, \eqref{svo4}, \eqref{svo5} and \eqref{latt} and in Lemma
\ref{L75}.
The $\sup_\Gamma$ means here the supremum over all
possible values of $z_j,z_j'\in\mathbb R^d $ (involved in the
definition of matrices $\mathbb B_t$ and $\mathbb V$) such that
%
%
\begin{equation}
\bigl\|\mathbb C^{-1/2}z_j-e_j\bigr\|\leq\delta,\qquad\bigl \|
\mathbb C^{-1/2}z_j'-e_j\bigr\|\leq
\delta\qquad\mbox{for } 1\leq j\leq s.\label{eq76f}
\end{equation}
Moreover, for any $b\in\mathbb R^d$, $r\ge1$ and
$\g>0$ and any fixed $t\in\mathbb R$ satisfying
$\llvert rt\rrvert\le c_s^* \sigma_1^{-2}$, where $c_s^*\ge1$ is
an arbitrary quantity depending on $s$ only, we have
%
%
\begin{equation}
\bigl| \widehat\Psi_b (t) \bigr| \ll_{\g, s} (p
N)^{-\g}+ r^{-s} \llvert t\rrvert^{-s/2} (\det\mathbb
C)^{-1/2}.\label{equ71p}
\end{equation}
\end{lemma}

Let $v=(m,\ov m)\in\mathbb R^{2s}$, $m,\ov m\in\mathbb R^{s}$ and
$t\in\mathbb R$. Then
%
%
\begin{equation}
\label{rho}\ov m+t m=\bigl(1 +t^2\bigr) \ov m+t (m-t \ov m).
\end{equation}
Equality \eqref{rho} implies that
%
%
\begin{equation}
\label{rho1} \Vert\ov m+t m\Vert\ll_s\Vert\ov m\Vert+\Vert m-t
\ov m\Vert\qquad\mbox{for } |t|\ll_s 1.
\end{equation}
Hence,
%
%
\begin{eqnarray}
\label{rho3} r \Vert m-t \ov m\Vert+ r^{-1} \Vert\ov m+t m\Vert
\ll_s r \Vert m-t \ov m\Vert+ r^{-1}\Vert\ov m\Vert
\nonumber
\\[-8pt]
\\[-8pt]
\eqntext{\mbox{for } r\gg1, |t|\ll_s 1.}
\end{eqnarray}
According to \eqref{svo4}--\eqref{svo5}, we
have
%
%
\begin{eqnarray}
\label{rhom} \mathbb D_r\mathbb U_tv&=&\bigl(r(m-t
\ov m), r^{-1}\ov m\bigr)\quad \mbox{and}
\nonumber
\\[-8pt]
\\[-8pt]
\nonumber
 \mathbb D_r
\mathbb K_tv&=&\bigl(r(m-t \ov m), r^{-1}(\ov m+tm )
\bigr).
\end{eqnarray}
It is clear that the operators
$\mathbb D_r\mathbb U_t$ and $\mathbb D_r\mathbb K_t$ are
invertible. Therefore, using \eqref{rho3} and \eqref{rhom} and
applying Lemmas \ref{Dav9} and \ref{Dav4}, we derive the
inequality
%
%
\begin{equation}
\label{kt6} \alpha(\mathbb D_{r}\mathbb U_t \Omega)
\ll_s\alpha(\mathbb D_{r}\mathbb K_t
\Omega) \qquad\mbox{for } r\gg1, |t|\ll_s 1,
\end{equation}
which is valid for any
lattice $\Omega\subset\mathbb R^{2s}$.

Let $\mathbb T$ be the permutation $(2s\times2s)$-matrix which
permutes the rows of a ${(2s\times2s)}$-matrix $\mathbb A$ so
that the new order (corresponding to the matrix $\mathbb T\mathbb
A$) is
\[
1,s+1,2,s+2,\ldots, s, 2s.
\]
Note that the operator
$\mathbb T$ is isometric and $\mathbb A\na\mathbb A \mathbb T^{-1}$
rearranges the columns of $\mathbb A$ in the order mentioned
above. It is easy to see that
%
%
\begin{equation}
\label{kt7} \alpha_j(\mathbb T \Omega) =\alpha_j(
\Omega), \qquad j=1,\ldots2s  \mbox{ and } \alpha(\mathbb T \Omega) =\alpha(
\Omega)
\end{equation}
for any lattice
$\Omega\subset\mathbb R^{2s}$.

Note now that
%
%
\begin{equation}
\label{kt8}\mathbb T\mathbb D_{r}\mathbb K_t
\Lambda_j =\mathbb T\mathbb D_{r}\mathbb K_t
\mathbb T^{-1}\mathbb T\Lambda_j =\mathbb
W_t\Delta_j,
\end{equation}
where
$\Delta_j$ is a lattice defined by
%
%
\begin{equation}
\label{kt9}\Delta_j=\mathbb T\Lambda_j
\end{equation}
and where $\mathbb W_t$ is $(2s\times2s)$-matrix
%
%
\begin{equation}
\label{latt34} \mathbb W_t= \pmatrix{\mathbb G_{r,t}&\mathbb O_2&\vdots&\mathbb
O_2
\vspace*{2pt}\cr
\mathbb O_2&\mathbb G_{r,t}&\vdots&\mathbb O_2
\vspace*{2pt}\cr
\cdots&\cdots&\cdots&\cdots
\vspace*{2pt}\cr
\mathbb O_2& \mathbb O_2 &\vdots& \mathbb G_{r,t}}
\end{equation}
constructed of $(2\times2)$-matrices $\mathbb O_2$ (with zero
entries) and
%
%
\begin{equation}
\mathbb G_{r,t} \= \pmatrix{ r &
-rt
\vspace*{2pt}\cr
r^{-1} t & r^{-1}}.
\end{equation}

Let $|t|\le2$ and
%
%
\begin{equation}
\label{kth} \theta=\arcsin\bigl(t \bigl(1 +t^2\bigr)^{-1/2}
\bigr)\qquad \mbox{or, equivalently } t=\tan\theta.
\end{equation}
Then we have
%
%
\begin{eqnarray}
\label{latw} |\theta|\le c^*&\=&\arcsin(2/\sqrt5), \qquad\cos\theta=\bigl(1
+t^2\bigr)^{-1/2},
\nonumber
\\[-8pt]
\\[-8pt]
\nonumber
\sin\theta&=&t \bigl(1 +t^2
\bigr)^{-1/2}.
\end{eqnarray}
It is easy to see that
%
%
\begin{eqnarray}
\label{latt1}\mathbb G_{r,t} = \bigl(1 +t^2
\bigr)^{1/2} \ov{\mathbb D}_{r} \ov{\mathbb
K}_\theta
\end{eqnarray}
and
%
%
\begin{equation}
\label{oot}\mathbb W_t=\bigl(1 +t^2
\bigr)^{1/2} \widetilde{\mathbb D}_r \widetilde{\mathbb
K}_\theta,
\end{equation}
where
%
%
\begin{equation}
\label{latt345}\qquad \widetilde{\mathbb D}_r=\pmatrix{\ov{\mathbb D}_{r}&\mathbb O_2&\vdots&
\mathbb O_2
\vspace*{2pt}\cr
\mathbb O_2&\ov{\mathbb D}_{r}&\vdots&\mathbb
O_2
\vspace*{2pt}\cr
\cdots&\cdots&\cdots&\cdots
\vspace*{2pt}\cr
\mathbb O_2& \mathbb O_2 &\vdots&\ov{\mathbb
D}_{r}}
 \quad\mbox{and}\quad \widetilde{
\mathbb K}_\theta= \pmatrix{\ov{\mathbb
K}_\theta&\mathbb O_2&\vdots&\mathbb O_2
\vspace*{2pt}\cr
\mathbb O_2&\ov{\mathbb K}_\theta&\vdots&\mathbb
O_2
\vspace*{2pt}\cr
\cdots&\cdots&\cdots&\cdots
\vspace*{2pt}\cr
\mathbb O_2& \mathbb O_2 &\vdots&\ov{\mathbb
K}_\theta}
\end{equation}
are $(2s\times2s)$-matrices with
%
%
\begin{equation}
\label{laz}\quad \ov{\mathbb D}_{r} \= \pmatrix{ r & 0
\vspace*{2pt}\cr
0 & r^{-1}}\quad \mbox{and}\quad \ov{
\mathbb K}_\theta\=\pmatrix{ \cos
\theta& -\sin\theta
\vspace*{2pt}\cr
\sin\theta&  \cos\theta
}\in\operatorname{SL}(2,\mathbb{R}).
\end{equation}
%
Substituting \eqref{oot} into equality \eqref{kt8}, we obtain
%
%
\begin{equation}
\label{kot8}\mathbb T\mathbb D_{r}\mathbb K_t\Lambda
_j =\bigl(1 +t^2\bigr)^{1/2} \widetilde{\mathbb
D}_r \widetilde{\mathbb K}_\theta\Delta_j.
\end{equation}

Below we also use the following crucial lemma of G\"otze and
Margulis (\citeyear{GotMar10}).

%
\begin{lemma}\label{GM} Let $\widetilde{\mathbb K}_\theta$ and
%
%
\begin{equation}
\label{laty} \widetilde{\mathbb H}= \pmatrix{
\ov{\mathbb H}&\mathbb O_2&\vdots&\mathbb O_2
\vspace*{2pt}\cr
\mathbb O_2&\ov{\mathbb H}&\vdots&\mathbb O_2
\vspace*{2pt}\cr
\cdots&\cdots&\cdots&\cdots
\vspace*{2pt}\cr
\mathbb O_2& \mathbb O_2 &\vdots&\ov{\mathbb H}}
\end{equation}
be $(2d\times2d)$-matrices such that $\ov{\mathbb H}\in
\mathcal G=\operatorname{SL}(2,\mathbb R)$ and $\widetilde{\mathbb
K}_\theta$ is defined in {\eqref{latt345}} and~{\eqref{laz}}. Let $\beta$ be a positive number such that
$\beta d>2$. Then, for any $\ov{\mathbb H}\in\mathcal G$ and
any lattice $\Delta\subset\mathbb R^{2d}$,
%
%
\begin{equation}
\int_{0}^{2\pi} \bigl(\alpha(\widetilde{\mathbb H}
\widetilde{\mathbb K}_\theta\Delta) \bigr)^\beta\,d\theta
\ll_{\beta,d} \bigl(\alpha(\Delta) \bigr)^\beta\Vert\ov
{\mathbb H}
\Vert^{\beta
d-2}.
\end{equation}
Here $\Vert\ov{\mathbb H}\Vert$ is the standard norm of
the linear operator $\ov{\mathbb H}\dvtx\R^2\to\R^2$.
\end{lemma}

Consider, under the conditions of Lemma \ref{GZ},
%
%
\begin{equation}
\label{IJ56} I_0 \= \int_{c_s\sigma_1^{-2}r^{-2+4/s}/2
}^{\sigma_1^{-2}/2}\bigl |
\widehat\Psi_b (t) \bigr| \frac{ dt }{
t } = \int_{c_s\sigma_1^{-2}r^{-2+4/s} }^{\sigma_1^{-2}}
\bigl| \widehat\Psi_b (t/2)\bigr | \frac{ dt }{ t }.
\end{equation}
By Lemma \ref{GZ}, we have
%
%
\begin{equation}
\label{IJ5} I_0 \ll_s (p N)^{-1}
\sigma_1^{s} (\det\mathbb C)^{-1/2}+
r^{- s/2} \sup_\Gamma J,
\end{equation}
where
%
%
\begin{equation}
\label{IJ6} J= \int_{r^{-1}}^1 \bigl(\alpha(\mathbb
D_r\mathbb U_t \Lambda) \bigr)^{1/2}
\frac{ dt }{ t } \le\sum_{j=2}^{\rho}I_{j},
\end{equation}
with
%
%
\begin{equation}
\label{IJ7} I_j\= \int_{j^{-1}}^{(j-1)^{-1}}
\bigl(\alpha(\mathbb D_r\mathbb U_t \Lambda)
\bigr)^{1/2} \frac{ dt }{ t },\qquad j=2,3,\ldots,\rho\=\lfloor r
\rfloor+1.
\end{equation}

Changing variable $t=vj^{-2}$ and $v=w+j$ in $I_j$ and using the
properties of matrices $\mathbb D_r$ and $\mathbb U_t$, we have
%
%
\begin{eqnarray}\label{IJ}
I_j&=&\int_{j}^{j^{2}(j-1)^{-1}} \bigl(\alpha(
\mathbb D_r\mathbb U_{vj^{-2}} \Lambda)
\bigr)^{1/2} \frac{ dv }{ v }
\nonumber
\\
& \le& \int_{j}^{j+2} \bigl(\alpha(\mathbb
D_r\mathbb U_{vj^{-2}} \Lambda) \bigr)^{1/2}
\frac{ dv }{ v }
\\
& =& \int_{0}^{2} \bigl(\alpha(\mathbb
D_r\mathbb U_{wj^{-2}}\mathbb U_{j^{-1}} \Lambda)
\bigr)^{1/2} \frac{ dw }{ w+j }.\nonumber
\end{eqnarray}
By \eqref{svo},
%
%
\begin{equation}
\label{IJ2} \mathbb D_r\mathbb U_{wj^{-2}} =\mathbb
D_{rj^{-1}}\mathbb D_{j}\mathbb U_{wj^{-2}} = \mathbb
D_{rj^{-1}} \mathbb U_{w} \mathbb D_{j}.
\end{equation}

According to \eqref{IJ} and \eqref{IJ2},
%
%
\begin{equation}
\label{IJ22} I_j\ll\frac{ 1 }{ j } \int_{0}^{2}
\bigl(\alpha(\mathbb D_{rj^{-1}}\mathbb U_t
\Lambda_{j}) \bigr)^{1/2} \,{dt},
\end{equation}
where the lattices
$\Lambda_j$ are defined in \eqref{jj}; see also \eqref{svo4},
\eqref{svo5} and \eqref{latt}. Using~\eqref{jj}, \eqref{latt36}
and \eqref{adas89}, we see that
%
%
\begin{equation}
\label{kth14} \alpha(\Lambda_{j}) \ll_s
\sigma_1^{2s} (\det\mathbb C)^{-1}.
\end{equation}

By \eqref{kt6}, \eqref{kt7} and \eqref{kot8}, we have
%
%
\begin{eqnarray}
\label{qkt} \alpha(\mathbb D_{rj^{-1}}\mathbb U_t
\Lambda_{j}) &\ll_s&\alpha(\mathbb D_{rj^{-1}}
\mathbb K_t \Lambda_{j}) =\alpha(\mathbb T\mathbb
D_{rj^{-1}}\mathbb K_t \Lambda_{j})
\nonumber
\\[-8pt]
\\[-8pt]
\nonumber
&\ll_s&\alpha(\widetilde{\mathbb D}_{rj^{-1}} \widetilde{
\mathbb K}_\theta\Delta_j)
\end{eqnarray}
for $|t|\ll_s1$, $r\ge1$,
$j=2,3,\ldots,\rho$, where $\Delta_j$ and $\theta$ are defined in
\eqref{kt9} and~\eqref{kth}, respectively. Using \eqref{kth},
\eqref{latw}, \eqref{latt345}, \eqref{qkt} and Lemma \ref{GM}
(with $d=s$), we obtain
%
%
\begin{eqnarray}
\label{kth23}\int_{0}^{2} \bigl(\alpha(\mathbb
D_{rj^{-1}}\mathbb U_t \Lambda_{j})
\bigr)^{1/2} {\,dt}&\ll_s&\int_{0}^{c^*}
\bigl(\alpha(\widetilde{\mathbb D}_{rj^{-1}} \widetilde{\mathbb
K}_\theta\Delta_j) \bigr)^{1/2} \frac{ d\theta}{ \cos^2\theta}
\nonumber
\\
&\ll&\int_{0}^{2\pi} \bigl(\alpha(\widetilde{
\mathbb D}_{rj^{-1}} \widetilde{\mathbb K}_\theta
\Delta_j) \bigr)^{1/2} {\,d\theta}
\\
&\ll_s& \Vert\ov{\mathbb D}_{rj^{-1}}\Vert^{s/2-2}
\bigl(\alpha(\Delta_j) \bigr)^{1/2},\nonumber
\end{eqnarray}
if $s\ge5$. It is clear
that $\Vert\ov{\mathbb D}_{rj^{-1}}\Vert= rj^{-1}$. Therefore,
according to \eqref{kt7}, \eqref{kt9}, \eqref{IJ22} and
\eqref{kth23},
%
%
\begin{equation}
\label{kth13}I_j \ll_s \frac{ 1 }{ j }
\bigl(rj^{-1}\bigr)^{s/2-2} \bigl(\alpha(\Lambda_{j})
\bigr)^{1/2}.
\end{equation}

By \eqref{IJ6}, \eqref{kth14} and \eqref{kth13}, we obtain,
for $s\ge5$,
%
%
\begin{equation}
\label{kth15} J \ll_s\sigma_1^{s} (\det
\mathbb C)^{-1/2}\sum_{j=2}^{\rho}
\frac
{ 1 }{ j } \bigl(rj^{-1}\bigr)^{s/2-2}
\ll_s r^{s/2-2} \sigma_1^{s} (\det
\mathbb C)^{-1/2}.
\end{equation}
By \eqref{dfn}, \eqref{defr}, \eqref{IJ5} and \eqref{kth15}, we
have $r\asymp_s (Np)^{1/2}$ and
%
%
\begin{eqnarray}
\label{kth16} I_0 \ll_s r^{-2}
\sigma_1^{s} (\det\mathbb C)^{-1/2}
\ll_s (Np)^{-1} \sigma_1^{s} (\det
\mathbb C)^{-1/2}.
\end{eqnarray}
It is clear that in a similar way we can establish that
%
%
\begin{eqnarray}
\label{k16} \int_{\sigma_1^{-2} }^{c(s)\sigma_1^{-2}} \bigl| \widehat
\Psi_b (t/2) \bigr| \frac{ dt }{ t } \ll_s r^{-2}
\sigma_1^{s} (\det\mathbb C)^{-1/2}
\ll_s (Np)^{-
1} \sigma_1^{s} (
\det\mathbb C)^{-1/2}\hspace*{-35pt}
\end{eqnarray}
for any quantity $c(s)$ depending on $s$ only. The proof
will be easier due to the fact that $t$ cannot be small in this
integral.

Thus, we have proved the following lemma.

%
\begin{lemma}\label{GZ2} Let the conditions of Lemma {\ref{L73}}
be satisfied with $s=d\ge5$, $\mathbb D=\mathbb C^{-1/2}$, $\delta
\leq
1/(5 s)$ and with an orthonormal system $\mathcal S=\mathcal
S_o=\{ {e}_{1},\ldots, {e}_{s} \}\subset\mathbb R^d$. Let $c_1(s)$
and $c_2(s)$ be some
quantities depending on $s$ only. Then there exists a $c_s$ such
that
%
%
\begin{equation}
\int_{c_1(s)\sigma_1^{- 2}r^{-2+4/s}
}^{c_2(s)\sigma_1^{-2}} \bigl| \widehat\Psi_b (t)
\bigr| \frac{
dt }{ t } \ll_s (Np)^{-1}
\sigma_1^{s} (\det\mathbb C)^{-1/2}, \label{eq71u}
\end{equation}
if $Np\gg_s c_s$, where $r$ is defined in \eqref{dfn} and
\eqref{defr}.
\end{lemma}

\section*{Acknowledgments} We would like to thank V.V. Ulyanov
for helpful discussions,
and two anonymous referees for useful suggestions which allowed us to
improve the presentation.

%
%



\printaddresses

\end{document}